\documentclass[11pt,reqno]{amsart}

\newtheorem{proposition}{Proposition}[section]
\newtheorem{lemma}[proposition]{Lemma}
\newtheorem{corollary}[proposition]{Corollary}
\newtheorem{theorem}[proposition]{Theorem}

\theoremstyle{definition}
\newtheorem{definition}[proposition]{Definition}
\newtheorem{example}[proposition]{Example}
\newtheorem{examples}[proposition]{Examples}
\newtheorem{remark}[proposition]{Remark}

\newcommand{\thlabel}[1]{\label{th:#1}}
\newcommand{\thref}[1]{Theorem~\ref{th:#1}}
\newcommand{\selabel}[1]{\label{se:#1}}
\newcommand{\seref}[1]{Section~\ref{se:#1}}
\newcommand{\lelabel}[1]{\label{le:#1}}
\newcommand{\leref}[1]{Lemma~\ref{le:#1}}
\newcommand{\prlabel}[1]{\label{pr:#1}}
\newcommand{\prref}[1]{Proposition~\ref{pr:#1}}
\newcommand{\colabel}[1]{\label{co:#1}}
\newcommand{\coref}[1]{Corollary~\ref{co:#1}}
\newcommand{\relabel}[1]{\label{re:#1}}
\newcommand{\reref}[1]{Remark~\ref{re:#1}}
\newcommand{\exlabel}[1]{\label{ex:#1}}
\newcommand{\exref}[1]{Example~\ref{ex:#1}}
\newcommand{\delabel}[1]{\label{de:#1}}
\newcommand{\deref}[1]{Definition~\ref{de:#1}}
\newcommand{\eqlabel}[1]{\label{eq:#1}}
\newcommand{\equref}[1]{(\ref{eq:#1})}

\def\ot{\otimes}

\newcommand{\Cc}{\mathcal{C}}

\def\*C{{}^*\hspace*{-1pt}{\Cc}}
\def\text#1{{\rm {\rm #1}}}

\input xy
\xyoption {all} \CompileMatrices

\usepackage{amssymb}
\usepackage{color,amssymb,graphicx,amscd,amsmath}
\usepackage[colorlinks,urlcolor=blue,linkcolor=blue,citecolor=blue]{hyperref}

\begin{document}

\title[Extending structures for Lie algebras]
{Extending structures for Lie algebras}

\author{A. L. Agore}
\address{Faculty of Engineering, Vrije Universiteit Brussel, Pleinlaan 2, B-1050 Brussels, Belgium}
\email{ana.agore@vub.ac.be and ana.agore@gmail.com}

\author{G. Militaru}
\address{Faculty of Mathematics and Computer Science, University of Bucharest, Str.
Academiei 14, RO-010014 Bucharest 1, Romania}
\email{gigel.militaru@fmi.unibuc.ro and gigel.militaru@gmail.com}

\thanks{A.L. Agore is research fellow ''Aspirant'' of FWO-Vlaanderen.
This work was supported by a grant of the Romanian National
Authority for Scientific Research, CNCS-UEFISCDI, grant no.
88/05.10.2011.}



\begin{abstract}
Let $\mathfrak{g}$ be a Lie algebra, $E$ a vector space containing
$\mathfrak{g}$ as a subspace. The paper is devoted to the
\emph{extending structures problem} which asks for the
classification of all Lie algebra structures on $E$ such that
$\mathfrak{g}$ is a Lie subalgebra of $E$. A general product,
called the unified product, is introduced as a tool for our
approach. Let $V$ be a complement of $\mathfrak{g}$ in $E$: the
unified product $\mathfrak{g} \,\natural \, V$ is associated to a
system $(\triangleleft, \, \triangleright, \, f, \{-, \, -\})$
consisting of two actions $\triangleleft$ and $\triangleright$, a
generalized cocycle $f$ and a twisted Jacobi bracket $\{-, \, -\}$
on $V$. There exists a Lie algebra structure $[-,-]$ on $E$
containing $\mathfrak{g}$ as a Lie subalgebra if and only if there
exists an isomorphism of Lie algebras $(E, [-,-]) \cong
\mathfrak{g} \,\natural \, V$. All such Lie algebra structures on
$E$ are classified by two cohomological type objects which are
explicitly constructed. The first one ${\mathcal
H}^{2}_{\mathfrak{g}} (V, \mathfrak{g})$ will classify all Lie
algebra structures on $E$ up to an isomorphism that stabilizes
$\mathfrak{g}$ while the second object ${\mathcal H}^{2} (V,
\mathfrak{g})$ provides the classification from the view point of
the extension problem. Several examples that compute both
classifying objects ${\mathcal H}^{2}_{\mathfrak{g}} (V,
\mathfrak{g})$ and ${\mathcal H}^{2} (V, \mathfrak{g})$ are worked
out in detail in the case of flag extending structures.
\end{abstract}

\maketitle

\section*{Introduction}
Lie algebras are studied in different fields such as differential
geometry, classical/quantum mechanics or the theory of particle
physics. In differential geometry, Lie algebras arise naturally on
the tangent space of symmetry (Lie) groups on manifolds. In
Hamiltonian mechanics the phase space is an example of a Lie
algebra while in quantum mechanics Heisenberg postulated the
existence of an infinite-dimensional Lie algebra of operators: the
theory of quantum mechanics follows more or less from properties
of Lie algebras. In the theory of particle physics Lie algebras
play a key role. For instance, bosonic string theory uses a Lie
algebra to formulate operators and the state space. Beyond the
remarkable applications in the above mentioned fields, Lie
algebras are objects of study in their own right. In this context
a natural question arises (throughout this paper, by 'an
isomorphism of Lie algebras $\varphi : E \to E$ that stabilizes
$\mathfrak{g}$' we mean an isomorphism of Lie algebras that acts
as the identity on the subspace $\mathfrak{g}$):

\textbf{Extending structures problem.} \textit{Let $\mathfrak{g}$
be a Lie algebra and $E$ a vector space containing $\mathfrak{g}$
as a subspace. Describe and classify up to an isomorphism of Lie
algebras that stabilizes $\mathfrak{g}$ the set of all Lie algebra
structures $[-, -]$ that can be defined on $E$ such that
$\mathfrak{g}$ is a Lie subalgebra of $(E, [-, -])$.}

We formulated and studied the same problem at the level of groups
in \cite{am-2010} and in a more general setting for Hopf algebras
in \cite{am-2011}. Even if the statement of the problem is
elementary, the problem turns out to be a difficult one. For
instance, if $\mathfrak{g} = \{0\}$ then the ES problem asks for
the classification of all Lie algebra structures on a given vector
space $E$, which is of course a wild problem. For this reason,
from now on we will assume that $\mathfrak{g} \neq \{0\}$.
Although the ES problem is very difficult, we can provide a
detailed answer to it in the case of what we call \emph{flag
extending structures} of $\mathfrak{g}$ to $E$ in the sense of
\deref{flagex}. To start with, we will explain what we mean by an
answer to the classification part of the ES problem. Having in
mind that we want to\emph{ extend} the Lie algebra structure on
$\mathfrak{g}$ to a bigger vector space, by classification we will
always mean classification up to an isomorphism of Lie algebras
$\varphi : E \to E$ that stabilizes $\mathfrak{g}$, i.e. $\varphi
(g) = g$, for all $g \in \mathfrak{g}$. Therefore, the problem
comes down to actually constructing the classifying object from
this first point of view: it will be a relative cohomological
'group'. On the other hand, as we shall explain below, the ES
problem generalizes the extension problem. Thus, we can also ask
for the classification from this point of view, i.e. up to an
isomorphism of Lie algebras that simultaneously stabilizes
$\mathfrak{g}$ and co-stabilizes $V$, its complement in $E$. This
will be the second classifying object which will generalize the
classical second cohomology group $H^{2} (V, \mathfrak{g})$.

The ES problem generalizes and unifies two famous problems in the
theory of Lie algebras: the \emph{extension problem} which goes
back to Chevalley and Eilenberg \cite{CE} and the
\emph{factorization problem} who's roots descend to the classical
results of Levi and Malcev \cite[Theorem 5]{bour}. We will explain
this briefly. Let $\mathfrak{g}$ and $\mathfrak{h}$ be two given
Lie algebras. The extension problem asks for the classification of
all Lie algebras $\mathfrak{E}$ which contain $\mathfrak{g}$ as an
ideal such that $\mathfrak{E}/\mathfrak{g} \cong \mathfrak{h}$.
Equivalently, the extension problem asks for the classification of
all Lie algebras $\mathfrak{E}$ that fit into an exact sequence of
Lie algebras
\begin{eqnarray} \eqlabel{extencros0}
\xymatrix{ 0 \ar[r] & \mathfrak{g} \ar[r]^{i} & \mathfrak{E}
\ar[r]^{\pi} & \mathfrak{h} \ar[r] & 0 }
\end{eqnarray}
Now, if in the ES problem we replace the condition
''$\mathfrak{g}$ is a Lie subalgebra of $(E, [-, -])$'' by a more
restrictive one, namely ''$\mathfrak{g}$ is an ideal of $E$'',
then what we obtain is in fact a reformulation of the extension
problem: any Lie algebra structure on $E$ containing
$\mathfrak{g}$ as an ideal is of course an extension of
$\mathfrak{g}$ through the Lie algebra $\mathfrak{h} :=
E/\mathfrak{g}$. In this case, let  $\pi : E \to \mathfrak{h}$ be
the canonical projection and $s : \mathfrak{h} \to E$ a linear
section of $\pi$, i.e. $\pi \circ s = {\rm Id}_{\mathfrak{h}}$. We
define the action $\triangleright $ and the cocycle $f$ by the
usual formulas:
\begin{eqnarray}
\triangleright : \mathfrak{h} \times \mathfrak{g} \to
\mathfrak{g}, \qquad x \triangleright g &:=& [s(x), \, g]  \eqlabel{cro1}\\
f: \mathfrak{h} \times \mathfrak{h} \to \mathfrak{g}, \qquad f(x,
y) &:=& [s(x), \, s(y)] - s \bigl( [x, \, y] \bigl) \eqlabel{cro2}
\end{eqnarray}
for all $x$, $y \in \mathfrak{h}$ and $g \in \mathfrak{g}$. In
geometrical language the action $\triangleright$ is called
\emph{connection}, while the cocycle $f$ is called
\emph{curvature}: for more details about the importance of the
extension problem in differential geometry we refer to
\cite[Section 4]{AMR2} and \cite{le}. Then the system
$(\mathfrak{g}, \, \mathfrak{h}, \, \triangleright, \, f )$ is a
crossed system of Lie algebras and the map
$$
\psi : \mathfrak{g} \#_{\triangleleft}^f \, \mathfrak{h} \to E,
\qquad \psi (g, x) := g + s(x)
$$
is an isomorphism of Lie algebras (see \coref{croslieide} for
details). In this classical reconstruction of a Lie algebra $E$
from an ideal and the corresponding quotient, the fact that
$\mathfrak{g}$ is an ideal of $E$ plays a crucial role from the
very beginning: namely it is the main ingredient in proving that
the action $\triangleright$ and the cocycle $f$ take values in
$\mathfrak{g} = {\rm Ker}(\pi)$. If we drop the assumption that
$\mathfrak{g}$ is an ideal of $E$ and we only ask for
$\mathfrak{g}$ to be a Lie subalgebra of $E$, as we formulated the
ES problem, then the above construction can not be performed
anymore and we have to come up with a new method of reconstructing
the Lie algebra $E$ from a given Lie subalgebra and another set of
data. This is what we do in \seref{unifiedprod}. For further
reference on the extension problem for Lie algebras, in the
abelian or non-abelian case, we refer to \cite{AMR1}, \cite{AMR2},
\cite{CE}, \cite{fa}, \cite{hu}, \cite{zu}.

The factorization problem is the dual of the extension problem. It
consists of describing and classifying up to an isomorphism all
Lie algebras $E$ that factorize through two given Lie algebras
$\mathfrak{g}$ and $\mathfrak{h}$: i.e. $E$ contains
$\mathfrak{g}$ and $\mathfrak{h}$ as Lie subalgabras such that $E
= \mathfrak{g} + \mathfrak{h}$ and $ \mathfrak{g} \cap
\mathfrak{h} = \{0\}$. The factorization problem is also a special
case of the ES problem, if we impose the following additional
assumption: we consider the complement $V$ of $\mathfrak{g}$ in
$E$ to be also a Lie subalgebra of $E$ isomorphic to
$\mathfrak{h}$. Dual to the extension problem, it was
independently proven in \cite{majid} and \cite{LW}, that a Lie
algebra $E$ factorizes through $\mathfrak{g}$ and $\mathfrak{h}$
if and only if $E \cong \mathfrak{g} \bowtie \mathfrak{h}$, where
$\mathfrak{g} \bowtie \mathfrak{h}$ is the bicrossed product
associated to a matched pair of Lie algebras $(\mathfrak{g},
\mathfrak{h}, \triangleleft, \triangleright)$. The details are
given in \seref{cazurispeciale}.

The paper is organized as follows: in \seref{unifiedprod} we will
perform the abstract construction of the \emph{unified product}
$\mathfrak{g} \,\natural \, V$: it is associated to a Lie algebra
$\mathfrak{g}$, a vector space $V$ and a system of data
$\Omega(\mathfrak{g}, V) = \bigl(\triangleleft, \, \triangleright,
\, f, \{-, \, -\} \bigl)$ called an extending datum of
$\mathfrak{g}$ through $V$. \thref{1} establishes the set of
axioms that has to be satisfied by $\Omega(\mathfrak{g}, V)$ such
that $\mathfrak{g} \,\natural \, V$ with a given canonical bracket
becomes a Lie algebra, i.e. is a unified product. In this case,
$\Omega(\mathfrak{g}, V) = \bigl(\triangleleft, \, \triangleright,
\, f, \{-, \, -\} \bigl)$ will be called a \emph{Lie extending
structure} of $\mathfrak{g}$ through $V$. Now let $\mathfrak{g}$
be a Lie algebra, $E$ a vector space containing $\mathfrak{g}$ as
a subspace and $V$ a given complement of $\mathfrak{g}$ in $E$.
\thref{classif} provides the answer to the description part of the
ES problem: there exists a Lie algebra structure $[-,-]$ on $E$
such that $\mathfrak{g}$ is a subalgebra of $(E, [-,-])$ if and
only if there exists an isomorphism of Lie algebras $(E, [-,-])
\cong \mathfrak{g} \,\natural \, V$, for some Lie extending
structure $\Omega(\mathfrak{g}, V) = \bigl(\triangleleft, \,
\triangleright, \, f, \{-, \, -\} \bigl)$ of $\mathfrak{g}$
through $V$. The answer to the classification part of the ES
problem is given in \thref{main1}: we will construct explicitly a
relative cohomology group, denoted by ${\mathcal
H}^{2}_{\mathfrak{g}} \, (V, \, \mathfrak{g} )$, which will be the
classifying object of all extending structures of the Lie algebra
$\mathfrak{g}$ to $E$ - the classification is given up to an
isomorphism of Lie algebras which stabilizes $\mathfrak{g}$.
Moreover, we also indicate the bijection between the elements of
${\mathcal H}^{2}_{\mathfrak{g}} \, (V, \, \mathfrak{g} )$ and the
isomorphism classes of all extending structures of $\mathfrak{g}$.
The construction of the second classifying object, denoted by
${\mathcal H}^{2} \, (V, \, \mathfrak{g} )$, is performed in
\reref{aldoileaob}: it parameterizes all extending structures of
$\mathfrak{g}$ to a Lie algebra on $E$ up to an isomorphism which
simultaneously stabilizes $\mathfrak{g}$ and co-stabilizes $V$ -
i.e. this classification is given from the point of view of the
extension problem. There exists a canonical projection ${\mathcal
H}^{2} \, (V, \, \mathfrak{g} ) \twoheadrightarrow {\mathcal
H}^{2}_{\mathfrak{g}} \, (V, \, \mathfrak{g})$ between these two
classifying objects. We point out that ${\mathcal H}^{2} \, (V, \,
\mathfrak{g} )$ generalizes the classical cohomology group ${\rm
H}^{2} \, (V, \, \mathfrak{g} )$: the latter is obtained as a
special case of ${\mathcal H}^{2} \, (V, \, \mathfrak{g} )$ if we
let the right action $\triangleleft$ to be the trivial one and the
extending structures of $\mathfrak{g}$ to be 'abelian' that is, if
we ask that the Lie algebra $\mathfrak{g}$ is contained in the
center of the unified products $\mathfrak{g} \,\natural \, V$. One
of the special cases that we introduce in \exref{twistedproduct}
is called \emph{twisted product}, the terminology being borrowed
from Hopf algebra theory. The two Lie algebras $\mathfrak{g}$ and
$V$ involved in the construction of the twisted product are
connected by a classical $2$-cocycle $f: V \times V \to
\mathfrak{g}$ and plays a key role in the classification of all
$6$-dimensional nilpotent Lie algebras \cite{gra}. Apart from the
twisted product, we show in \seref{cazurispeciale} that both the
classical crossed product and bicrossed product of Lie algebras
appear as special cases of the unified product.

\thref{main1} offers the theoretical answer to the extending
structures problem. The challenge we are left to deal with is a
purely computational one: for a given Lie algebra $\mathfrak{g}$
that is a subspace in a vector space $E$ with a given complement
$V$ we have to compute explicitly the classifying object
${\mathcal H}^{2}_{\mathfrak{g}} \, (V, \, \mathfrak{g} )$ and
then to list the set of types of all Lie algebra structures on $E$
which extend the Lie algebra structure on $\mathfrak{g}$. This is
highly nontrivial considering that the construction of ${\mathcal
H}^{2}_{\mathfrak{g}} \, (V, \, \mathfrak{g} )$ is very laborious.
In \seref{exemple} we shall identify a way of computing ${\mathcal
H}^{2}_{\mathfrak{g}} \, (V, \, \mathfrak{g} )$ for the case when
the complement $V$ is finite dimensional: namely for those that
are \emph{flag extending structures} of $\mathfrak{g}$ to $E$ in
the sense of \deref{flagex}. All flag extending structures of
$\mathfrak{g}$ to $E$ can be completely described by a recursive
reasoning where the key step is the case when $\mathfrak{g}$ has
codimension $1$ as a subspace of $E$. This case is completely
solved in \thref{clasdim1} where ${\mathcal H}^{2}_{\mathfrak{g}}
(V, \mathfrak{g} )$ and ${\mathcal H}^{2} \, (V, \mathfrak{g} )$
are completely described: both objects are quotient pointed sets
of the set ${\rm TwDer} (\mathfrak{g})$ of all twisted derivations
of $\mathfrak{g}$ introduced in \deref{lambdaderivariii}. The set
${\rm TwDer} (\mathfrak{g})$ contains the usual space of
derivations ${\rm Der} (\mathfrak{g})$ via the canonical embedding
which is an isomorphism in the case when $\mathfrak{g}$ is a
perfect Lie algebra. Finally, two explicit examples are given in
\exref{ultimulexperfect} and \exref{cazulneperfect}: in the first
case all extending structures of a $5$-dimensional perfect Lie
algebra to a space of dimension $6$ are classified while in the
second one we list all types of extending structures of the
non-perfect Lie algebra $\textsf{gl}(2, k)$ to a space of
dimension $5$.

\section{Preliminaries}\selabel{prel}
Throughout this paper $k$ will be a field. All vector spaces, Lie
algebras, linear or bilinear maps are over $k$. A map $f: V \to W$
between two vector spaces is called the \emph{trivial map} if $f
(v) = 0$, for all $v\in V$. Let $\mathfrak{g} \leq E$ be a
subspace in a vector space $E$; a subspace $V$ of $E$ such that $E
= \mathfrak{g} + V$ and $V \cap \mathfrak{g} = 0$ is called a
complement of $\mathfrak{g}$ in $E$. Such a complement is unique
up to an isomorphism and its dimension is called the codimension
of $\mathfrak{g}$ in $E$. We recall briefly the basic concepts
related to Lie algebras; for all unexplained notations or
definitions we refer the reader to \cite{bour}, \cite{EW} or
\cite{H}. A Lie algebra is a vector space $\mathfrak{g}$, together
with a bilinear map $[- , \, -] : \mathfrak{g} \times \mathfrak{g}
\to \mathfrak{g}$ called bracket satisfying the following two
properties:
$$
[g, \, g] = 0, \qquad [g, \, [h, \,l] ] + [h, \, [l, \, g] ] + [l,
\, [g, \, h] ] = 0
$$
for all $g$, $h$, $l\in \mathfrak{g}$. The second condition is
called the Jacobi identity. Let $\mathfrak{g}$ be a Lie algebra
and $\mathfrak{g}' := [ \mathfrak{g} , \mathfrak{g}]$ be the
derived algebra of $\mathfrak{g}$; $\mathfrak{g}$ is called
perfect if $\mathfrak{g}' = \mathfrak{g}$ and abelian if
$\mathfrak{g}' = 0$. Representations of a Lie algebra
$\mathfrak{g}$ will be viewed as modules over $\mathfrak{g}$;
moreover, we shall work with both concepts of right and left
$\mathfrak{g}$-modules. Explicitly, a \emph{right
$\mathfrak{g}$-module} is a vector space $V$ together with a
bilinear map $ \triangleleft : V \times \mathfrak{g} \to V$,
called a right action of $\mathfrak{g}$ on $V$, satisfying the
following compatibility
\begin{equation}\eqlabel{moduldrept}
x \triangleleft [g, \, h] = (x \triangleleft g) \triangleleft h -
(x \triangleleft h) \triangleleft g
\end{equation}
for all $x\in V$ and $g$, $h \in \mathfrak{g}$. A \emph{left
$\mathfrak{g}$-module} is a vector space $V$ together with a
bilinear map $ \triangleright : \mathfrak{g} \times V \to V$,
called a left action of $\mathfrak{g}$ on $V$ such that:
\begin{equation}\eqlabel{modulstring}
[g, \, h] \triangleright x = g \triangleright (h \triangleright x)
- h \triangleright (g \triangleright x)
\end{equation}
for all $g$, $h \in \mathfrak{g}$ and $x\in V$. Any right
$\mathfrak{g}$-module is a left $\mathfrak{g}$-module via $g
\triangleright x := - x \triangleleft g$ and viceversa, that is
the category of right $\mathfrak{g}$-modules is isomorphic to the
category of left $\mathfrak{g}$-modules and both of them are
isomorphic to the category of representations of $\mathfrak{g}$.
${\rm Der} (\mathfrak{g})$ denotes the Lie algebra of all
derivations of $\mathfrak{g}$, that is all linear maps $D:
\mathfrak{g} \to \mathfrak{g}$ such that
$$
D ([g, \, h]) = [D(g), \, h] + [g, \, D(h)]
$$
for all $g$, $h\in \mathfrak{g}$. ${\rm Der} (\mathfrak{g})$ is a
Lie algebra with the bracket $[D_1, \, D_2] := D_1\circ D_2 - D_2
\circ D_1$ and the map
$$
{\rm ad} : \mathfrak{g} \to {\rm Der} (\mathfrak{g}), \quad {\rm
ad} (g) := [g, \, -] : \mathfrak{g} \to \mathfrak{g}, \quad h
\mapsto [g, \, h]
$$
is called the adjoint representation of $\mathfrak{g}$. Then,
${\rm Ker} ({\rm ad}) = Z (\mathfrak{g})$, the center of
$\mathfrak{g}$, and ${\rm Im} ({\rm ad})$ is called the space of
inner derivation of $\mathfrak{g}$ and will be denoted by ${\rm
Inn}(\mathfrak{g})$. ${\rm Inn}(\mathfrak{g} )$ is a Lie ideal in
${\rm Der}(\mathfrak{g})$ and
$$
{\rm Out} (\mathfrak{g}) := {\rm Der} (\mathfrak{g}) / {\rm
Inn}(\mathfrak{g} )
$$
is the Lie algebra of outer derivations of $\mathfrak{g}$. If
$\mathfrak{g}$ is semisimple, then $\mathfrak{g}$  is perfect,
${\rm Inn}(\mathfrak{g}) = {\rm Der} (\mathfrak{g})$ and
$Z(\mathfrak{g}) = 0$ (\cite{H}).

In order to answer the classification part of the extending
structures problem we need to introduce the following:

\begin{definition} \delabel{echivextedn}
Let $\mathfrak{g}$ be a Lie algebra, $E$ a vector space such that
$\mathfrak{g}$ is a subspace of $E$ and $V$ a complement of
$\mathfrak{g}$ in $E$. For a linear map $\varphi: E \to E$ we
consider the diagram:
\begin{eqnarray} \eqlabel{diagrama}
\xymatrix {& \mathfrak{g} \ar[r]^{i} \ar[d]_{Id} & {E}
\ar[r]^{\pi} \ar[d]^{\varphi} & V \ar[d]^{Id}\\
& \mathfrak{g} \ar[r]^{i} & {E}\ar[r]^{\pi } & V}
\end{eqnarray}
where $\pi : E \to V$ is the canonical projection of $E =
\mathfrak{g} + V$ on $V$ and $i: \mathfrak{g} \to E$ is the
inclusion map. We say that $\varphi: E \to E$ \emph{stabilizes}
$\mathfrak{g}$ (resp. \emph{co-stabilizes} $V$) if the left square
(resp. the right square) of the diagram \equref{diagrama} is
commutative.

Let $\{-, \, -\}$ and $\{-, \, -\}'$ be two Lie algebra structures
on $E$ both containing $\mathfrak{g}$ as a Lie subalgebra. $\{-,
\, -\}$ and $\{-, \, -\}'$ are called \emph{equivalent}, and we
denote this by $(E, \{-, \, -\}) \equiv (E, \{-, \, -\}')$, if
there exists a Lie algebra isomorphism $\varphi: (E, \{-, \, -\})
\to (E, \{-, \, -\}')$ which stabilizes $\mathfrak{g}$.

$\{-, \, -\}$ and $\{-, \, -\}'$ are called \emph{cohomologous},
and we denote this by $(E, \{-, \, -\}) \approx (E, \{-, \,
-\}')$, if there exists a Lie algebra isomorphism $\varphi: (E,
\{-, \, -\}) \to (E, \{-, \, -\}')$ which stabilizes
$\mathfrak{g}$ and co-stabilizes $V$, i.e. the diagram
\equref{diagrama} is commutative.
\end{definition}

$\equiv$ and $\approx$ are both equivalence relations on the set
of all Lie algebras structures on $E$ containing $\mathfrak{g}$ as
a Lie subalgebra and we denote by ${\rm Extd} \, (E,
\mathfrak{g})$ (resp. ${\rm Extd}' \, (E, \mathfrak{g})$) the set
of all equivalence classes via $\equiv$ (resp. $\approx$). Thus,
${\rm Extd} \, (E, \mathfrak{g})$ is the classifying object of the
extending structures problem: by explicitly computing ${\rm Extd}
\, (E, \mathfrak{g})$ we obtain a parametrization of the set of
all isomorphism classes of Lie algebra structures on $E$ that
stabilizes $\mathfrak{g}$. ${\rm Extd}' \, (E, \mathfrak{g})$
gives a classification of the ES problem from the point of view of
the extension problem. Any two cohomologous brackets on $E$ are of
course equivalent, hence there exists a canonical projection
$$
{\rm Extd}' \,  (E, \mathfrak{g}) \twoheadrightarrow {\rm Extd} \,
(E, \mathfrak{g})
$$
The classification part of the extending structures problem will
be solved by computing explicitly both classifying objects.
Borrowing the terminology from Lie algebra cohomology, we will see
that ${\rm Extd}' \, (E, \mathfrak{g})$ is parameterized by a
cohomological object denoted by ${\mathcal H}^{2} \, (V, \,
\mathfrak{g})$, which will be explicitly constructed and which
generalizes the classical second cohomology group for Lie algebras
\cite{CE}, while ${\rm Extd} \, (E, \mathfrak{g})$ will be
parameterized by a relative cohomological object, denoted by
${\mathcal H}^{2}_{\mathfrak{g}} \, (V, \, \mathfrak{g} )$ which
turns out to be a quotient of ${\mathcal H}^{2} \, (V, \,
\mathfrak{g})$.

\section{Unified products for Lie algebras}\selabel{unifiedprod}

\begin{definition}\delabel{exdatum}
Let $\mathfrak{g}$ be a Lie algebra and $V$ a vector space. An
\textit{extending datum of $\mathfrak{g}$ through $V$} is a system
$\Omega(\mathfrak{g}, V) = \bigl(\triangleleft, \, \triangleright,
\, f, \{-, \, -\} \bigl)$ consisting of four bilinear maps
$$
\triangleleft : V \times \mathfrak{g} \to V, \quad \triangleright
: V \times \mathfrak{g} \to \mathfrak{g}, \quad f: V\times V \to
\mathfrak{g}, \quad \{-, \, -\} : V\times V \to V
$$
Let $\Omega(\mathfrak{g}, V) = \bigl(\triangleleft, \,
\triangleright, \, f, \{-, \, -\} \bigl)$ be an extending datum.
We denote by $ \mathfrak{g} \, \,\natural \,_{\Omega(\mathfrak{g},
V)} V = \mathfrak{g} \, \,\natural \, V$ the vector space
$\mathfrak{g} \, \times V$ with the bilinear map $[ -, \, -] :
(\mathfrak{g} \times V) \times (\mathfrak{g} \times V) \to
\mathfrak{g} \times V$ defined by:
\begin{equation}\eqlabel{brackunif}
[(g, x), \, (h, y)] := \bigl( [g, \, h] + x \triangleright h -
y\triangleright g + f(x, y), \,\, \{x, \, y \} + x\triangleleft h
- y\triangleleft g \bigl)
\end{equation}
for all $g$, $h \in \mathfrak{g}$ and $x$, $y \in V$. The object
$\mathfrak{g} \,\natural \, V$ is called the \textit{unified
product} of $\mathfrak{g}$ and $\Omega(\mathfrak{g}, V)$ if it is
a Lie algebra with the bracket given by \equref{brackunif}. In
this case the extending datum $\Omega(\mathfrak{g}, V) =
\bigl(\triangleleft, \, \triangleright, \, f, \{-, \, -\} \bigl)$
is called a \textit{Lie extending structure} of $\mathfrak{g}$
through $V$. The maps $\triangleleft$ and $\triangleright$ are
called the \textit{actions} of $\Omega(\mathfrak{g}, V)$ and $f$
is called the \textit{cocycle} of $\Omega(\mathfrak{g}, V)$.
\end{definition}

The extending datum $\Omega(\mathfrak{g}, V) =
\bigl(\triangleleft, \, \triangleright, \, f, \{-, \, -\} \bigl)$,
for which $(\triangleleft, \, \triangleright, \, f, \{-, \, -\}
\bigl)$ are all the trivial maps is an example of a Lie extending
structure, called the \emph{trivial extending structure} of
$\mathfrak{g}$ through $V$. Let $\Omega(\mathfrak{g}, V)$ be an
extending datum of $\mathfrak{g}$ through $V$. Then, the following
relations, very useful in computations, hold in $\mathfrak{g}
\,\natural \, V$:
\begin{eqnarray}
[(g, 0), \, (h, y)] &=& \bigl([g, \, h] - y \triangleright
g, \, - y \triangleleft g \bigl) \eqlabel{001}\\
\left[(0, x), \, (h, y)\right] &=& \bigl( x \triangleright h +
f(x, y), \, x \triangleleft h + \{x, \, y\} \bigl) \eqlabel{002}
\end{eqnarray}
for all $g$, $h \in \mathfrak{g}$ and $x$, $y \in V$.

\begin{theorem}\thlabel{1}
Let $\mathfrak{g}$ be a Lie algebra, $V$ a $k$-vector space and
$\Omega(\mathfrak{g}, V)$ an extending datum of $\mathfrak{g}$ by
$V$. The following statements are equivalent:

$(1)$ $\mathfrak{g} \,\natural \, V$ is a unified product;

$(2)$ The following compatibilities hold for any $g$, $h \in
\mathfrak{g}$, $x$, $y$, $z \in V$:
\begin{enumerate}
\item[(LE1)] $f(x,\, x) = 0$, \qquad $\{x,\, x \} = 0$; \\
\item[(LE2)] $(V, \, \triangleleft)$ is a right $\mathfrak{g}$-module;\\
\item[(LE3)] $x \triangleright [g, \, h] = [x \triangleright g, \,
h] + [g, \, x \triangleright h] + (x \triangleleft g)
\triangleright h - (x \triangleleft h) \triangleright g$;\\
\item[(LE4)] $\{x,\, y \} \triangleleft g = \{x,\, y \triangleleft
g\} + \{x \triangleleft g,\, y \} + x \triangleleft (y
\triangleright g)
- y \triangleleft (x \triangleright g)$;\\
\item[(LE5)] $\{x,\, y \} \triangleright g = x \triangleright (y
\triangleright g) - y \triangleright (x \triangleright g) + [g, \,
f(x,\, y)] + f(x, y \triangleleft g) + f(x \triangleleft g, y)$;\\
\item[(LE6)] $f\bigl(x, \{y,\, z \}\bigl) + f\bigl(y, \{z,\, x
\}\bigl) + f\bigl(z, \{x,\, y \}\bigl) + x \triangleright f(y,  z)
+ y \triangleright f(z,  x) + z \triangleright f(x,
y) = 0$;\\
\item[(LE7)] $\{x, \, \{y, \,z\}\} + \{y, \, \{z, \,x\}\} + \{z,
\, \{x, \,y\}\} + x \triangleleft f(y, z) + y \triangleleft f(z,
x) + z \triangleleft f(x, y) = 0$.
\end{enumerate}
\end{theorem}

Before going into the proof of the theorem, we make a few remarks
on the compatibilities in \thref{1}. Aside from the fact that $V$
 is not a Lie algebra, $(LE3)$ and $(LE4)$ are exactly the
compatibilities defining a matched pair of Lie algebras
\cite[Definition 8.3.1]{majid2}. The compatibility condition
$(LE5)$ is called the \emph{twisted module condition} for the
action $\triangleright$; in the case when $V$ is a Lie algebra it
measures how far $(\mathfrak{g}, \triangleright)$ is from being a
left $V$-module. $(LE6)$ is called the \emph{twisted cocycle
condition}: if $\triangleright$ is the trivial action and $(V,
\{-, \, - \})$ is a Lie algebra then the compatibility condition
$(LE6)$ is exactly the classical $2$-cocycle condition for Lie
algebras. $(LE7)$ is called the \emph{twisted Jacobi condition}:
it measures how far $\{-, \, -\}$ is from being a Lie structure on
$V$. If either $\triangleleft$ or $f$ is the trivial map, then
$(LE7)$ is equivalent to $\{-, \, -\}$ being a Lie bracket on $V$.

\begin{proof}
For any $g \in \mathfrak{g}$ and $x \in V$ we have:
\begin{eqnarray*}
[(g, x), \, (g, x)] &=& \bigl([g, \, g] + x \triangleright g - x
\triangleright g + f(x, x), \, \{x, \, x\} + x \triangleleft g - x
\triangleleft g\bigl)\\
&=& \bigl(f(x, x), \, \{x, \, x\}\bigl)
\end{eqnarray*}
Therefore, $[(g, x), \, (g, x)] = 0$ if and only if $(LE1)$ holds.
From now on we will assume that $(LE1)$ holds. In particular, we
have that $f(x, y) = - f(y, x)$ and $\{x, y\} = - \{y, x\}$, for
all $x$, $y \in V$ since $f$ and $\{-, -\}$ are bilinear maps.
Thus $\mathfrak{g} \,\natural \, V$ is a Lie algebra if and only
if Jacobi's identity holds, i.e.:
\begin{equation}\eqlabel{005}
\bigl[(g, x), \, [(h, y), \, (l, z)]\bigl] + \bigl[(h, y), \, [(l,
z), \, (g, x)]\bigl] + \bigl[(l, z), \, [(g, x), \, (h, y)]\bigl]
= 0
\end{equation}
for all $g$, $h$, $l \in \mathfrak{g}$ and $x$, $y$, $z \in V$.
Since in $\mathfrak{g} \,\natural \, V$ we have $(g, x) = (g, 0) +
(0, x)$ it follows that \equref{005} holds if and only if it holds
for all generators of $\mathfrak{g} \,\natural \, V$, i.e. the set
$\{(g, \, 0) ~|~ g \in \mathfrak{g}\} \cup \{(0, \, x) ~|~ x \in
V\}$. Since \equref{005} is invariant under circular permutations
we are left with only three cases to study. First, we should
notice that \equref{005} holds for the triple $(g, 0)$, $(h, 0)$,
$(l, 0)$ as we have:
\begin{eqnarray*}
&&\bigl[(g, 0), \, [(h, 0), \, (l, 0)]\bigl] + \bigl[(h, 0), \,
[(l,
0), \, (g, 0)]\bigl] + \bigl[(l, 0), \, [(g, 0), \, (h, 0)]\bigl] =\\
&& = \bigl([g, \, [h, \,l] ] + [h, \, [l, \, g] ] + [l, \, [g, \,
h] ], 0\bigl) = (0, 0)
\end{eqnarray*}
Next, we prove that \equref{005} holds for $(g, 0)$, $(h, 0)$,
$(0, x)$ if and only if $(LE2)$ and $(LE3)$ hold. Indeed, we have:
\begin{eqnarray*}
&&\bigl[(g, 0), \, [(h, 0), \, (0, x)]\bigl] + \bigl[(h, 0), \,
[(0, x), \, (g, 0)]\bigl] + \bigl[(0, x), \, [(g, 0), \, (h, 0)]\bigl] = \\
&\stackrel{\equref{001}, \equref{002}} {=}& \left[(g, 0), \, (- x
\triangleright h, - x \triangleleft h)\right] + \left[(h, 0), \,
(x \triangleright g, x \triangleleft g)\right] + \left[(0, x),\,
([g, h], 0) \right]\\
&\stackrel{\equref{001}, \equref{002}} {=}&
\bigl(- [g, \, x \triangleright h] + (x \triangleleft h)
\triangleright g
 + [h, \, x \triangleright g] - (x \triangleleft g) \triangleright h +
 x \triangleright [g, \, h],\\
&&(x \triangleleft h) \triangleleft g - (x \triangleleft g)
\triangleleft h + x \triangleleft [g,\, h] \bigl)
\end{eqnarray*}
Thus we proved that \equref{005} holds for $(g, 0)$, $(h, 0)$,
$(0, x)$ if and only if $(LE2)$ and $(LE3)$ hold. Now, we prove
that \equref{005} holds for $(g, 0)$, $(0, x)$, $(0, y)$ if and
only if $(LE4)$ and $(LE5)$ hold. Indeed, we have:
\begin{eqnarray*}
&&\bigl[(g, 0), \, [(0, x), \, (0, y)]\bigl] + \bigl[(0, x), \,
[(0, y), \, (g, 0)]\bigl] + \bigl[(0, y), \, [(g, 0), \, (0, x)]\bigl] = \\
&\stackrel{\equref{001}, \equref{002}} {=}& \bigl[(g, 0), \,
\bigl(f(x, y), \{x, \, y\}\bigl)\bigl] + [(0, x), \, (y
\triangleright g, y \triangleleft g)]
 + [(0, y), \, (- x \triangleright g, - x \triangleleft g)]\\
&\stackrel{\equref{001}, \equref{002}} {=}& \bigl([g,\, f(x, y)] -
\{x,\, y\} \triangleright g, - \{x,\, y\} \triangleleft g\bigl) +
\bigl(x \triangleright (y \triangleright g) + f(x, y \triangleleft g), \,
x \triangleleft (y \triangleright g) +\\
&& \{x,\, y \triangleleft g\}\bigl) + \bigl(y \triangleright (- x
\triangleright g)
+ f(y, - x \triangleleft g), y \triangleleft (- x \triangleright g)
+ \{y,\, - x \triangleleft g\}\bigl)\\
&{=}& \bigl([g,\, f(x, y)] - \{x,\, y\} \triangleright g + x
\triangleright (y \triangleright g) + f(x, y \triangleleft g) -
y \triangleright (x \triangleright g) - f(y, x \triangleleft g), \\
&& - \{x,\, y\} \triangleleft g + x \triangleleft (y
\triangleright g) + \{x,\, y \triangleleft g\} - y \triangleleft
(x \triangleright g) - \{y,\, x \triangleleft g\}\bigl)
\end{eqnarray*}
Therefore, having in mind that for all $x$, $y \in V$ we have
$f(x, y) = - f(y, x)$ and $\{x, y\} = - \{y, x\}$ it follows that
\equref{005} holds for $(g, 0)$, $(0, x)$ $(0, y)$ if and only if
$(LE4)$ and $(LE5)$ hold. Finally, we will prove that \equref{005}
holds for $(0, x)$, $(0, y)$, $(0, z)$ if and only if $(LE6)$ and
$(LE7)$ hold. Indeed, we have:
\begin{eqnarray*}
&&\bigl[(0, x), \, [(0, y), \, (0, z)]\bigl] + \bigl[(0, y), \,
[(0, z), \, (0, x)]\bigl] + \bigl[(0, z), \, [(0, x), \, (0, y)]\bigl] = \\
&{=}& \bigl[(0, x),\, (f(y, z), \{y, z\})\bigl] + \bigl[(0, y), \,
(f(z, x), \{z, x\})  \bigl] + \bigl[(0, z),\, (f(x, y), \{x, y\})\bigl]\\
&\stackrel{\equref{002}}{=}& \bigl(x \triangleright f(y, z) + f(x,
\{y,\, z\}),\, x \triangleleft f(y, z) + \{x, \{y,\, z\}\}\bigl) +
\bigl(y \triangleright f(z, x) + f(y, \{z,\, x\}), \\
&& y \triangleleft f(z, x) + \{y, \{z,\, x\}\}\bigl) + \bigl(z
\triangleright f(x, y) + f(z, \{x,\, y\}),\, z \triangleleft f(x,
y) + \{z, \{x,\, y\}\}\bigl)
\end{eqnarray*}
Thus, \equref{005} holds for $(0, x)$, $(0, y)$, $(0, z)$ if and
only if $(LE6)$ and $(LE7)$ hold and the proof is finished.
\end{proof}

From now on, in light of \thref{1}, a Lie extending structure of
$\mathfrak{g}$ through $V$ will be viewed as a system
$\Omega(\mathfrak{g}, V) = \bigl(\triangleleft, \, \triangleright,
\, f, \{-, \, -\} \bigl)$ satisfying the compatibility conditions
$(LE1)-(LE7)$. We denote by ${\mathcal L} (\mathfrak{g}, V)$ the
set of all Lie extending structures of $\mathfrak{g}$ through $V$.

\begin{example}\exlabel{twistedproduct}
We provide the first example of a Lie extending structure and the
corresponding unified product. More examples will be given in
\seref{cazurispeciale} and \seref{exemple}.

Let $\Omega(\mathfrak{g}, V) = \bigl(\triangleleft, \,
\triangleright, \, f, \{-, \, -\} \bigl)$ be an extending datum of
a Lie algebra $\mathfrak{g}$ through a vector space $V$ such that
$\triangleleft$ and $\triangleright$ are both trivial maps, i.e.
$x\triangleleft g = x \triangleright g = 0$, for all $x$, $y\in V$
and $g \in \mathfrak{g}$. Then, $\Omega(\mathfrak{g}, V) =
\bigl(f, \{-, \, -\} \bigl)$ is a Lie extending structure of
$\mathfrak{g}$ through $V$ if and only if $(V, \{-, -\})$ is a Lie
algebra and $f: V \times V \to \mathfrak{g}$ is a classical
$2$-cocycle, that is:
\begin{equation*}\eqlabel{2cociclucl}
f(x, x) = 0, \qquad \left[g, \, f(x, y)\right] = 0, \qquad
f\bigl(x, \{y,\, z \}\bigl) + f\bigl(y, \{z,\, x \}\bigl) +
f\bigl(z, \{x,\, y \}\bigl)  = 0
\end{equation*}
for all $g \in \mathfrak{g}$, $x$, $y$, $z \in V$. In this case,
the associated unified product $\mathfrak{g} \,\natural
\,_{\Omega(\mathfrak{g}, V)} V$ will be denoted by $\mathfrak{g}
\#^f \, V $ and we shall call it the \emph{twisted product} of the
Lie algebras $\mathfrak{g}$ and $V$. Hence, the twisted product
associated to a given $2$-cocycle $f: V \times V \to \mathfrak{g}$
between Lie algebras is the vector space $\mathfrak{g} \times \, V
$ with the bracket given for any $g$, $h \in \mathfrak{g}$ and
$x$, $y \in V$ by:
\begin{equation}\eqlabel{twistedproduct}
[(g, x), \, (h, y)] := \bigl( [g, \, h] + f(x, y), \, \{x, \, y \}
\bigl)
\end{equation}
The twisted product of two Lie algebras plays the crucial role in
the classification of all $6$-dimensional nilpotent Lie algebras
given in \cite{gra}.
\end{example}

Let $\Omega(\mathfrak{g}, V) = \bigl(\triangleleft, \,
\triangleright, \, f, \{-, \, -\} \bigl) \, \in {\mathcal L}
(\mathfrak{g}, V)$ be a Lie extending structure and $\mathfrak{g}
\,\natural \, V$ the associated unified product. Then the
canonical inclusion
$$
i_{\mathfrak{g}}: \mathfrak{g} \to \mathfrak{g} \,\natural \, V,
\qquad i_{\mathfrak{g}}(g) = (g, \, 0)
$$
is an injective Lie algebra map. Therefore, we can see
$\mathfrak{g}$ as a Lie subalgebra of $\mathfrak{g} \,\natural \,
V$ through the identification $\mathfrak{g} \cong
i_{\mathfrak{g}}(\mathfrak{g}) \cong \mathfrak{g} \times \{0\}$.
Conversely, we will prove that any Lie algebra structure on a
vector space $E$ containing $\mathfrak{g}$ as a Lie subalgebra is
isomorphic to a unified product. In this way, we obtain the answer
to the description part of the extending structures problem:

\begin{theorem}\thlabel{classif}
Let $\mathfrak{g}$ be a Lie algebra, $E$ a vector space containing
$\mathfrak{g}$ as a subspace and $[-, \,-]$ a Lie algebra
structure on $E$ such that $\mathfrak{g}$ is a Lie subalgebra in
$(E, [-, \,-])$. Then there exists a Lie extending structure
$\Omega(\mathfrak{g}, V) = \bigl(\triangleleft, \, \triangleright,
\, f, \{-, \, -\} \bigl)$ of $\mathfrak{g}$ trough a subspace $V$
of $E$ and an isomorphism of Lie algebras $(E, [-, \,-]) \cong
\mathfrak{g} \,\natural \, V$ that stabilizes $\mathfrak{g}$ and
co-stabilizes $V$.
\end{theorem}

\begin{proof} As $k$ is a field, there exists a linear map $p: E \to
\mathfrak{g}$ such that $p(g) = g$, for all $g \in \mathfrak{g}$.
Then $V := \rm{ker}(p)$ is a subspace of $E$ and a complement of
$\mathfrak{g}$ in $E$. We define the extending datum of
$\mathfrak{g}$ through $V$ by the following formulas:
\begin{eqnarray*}
\triangleright = \triangleright_p : V \times \mathfrak{g} \to
\mathfrak{g}, \qquad x \triangleright g &:=& p \bigl([x, \,g]\bigl)\\
\triangleleft = \triangleleft_p: V \times \mathfrak{g} \to V,
\qquad x \triangleleft g &:=& [x, \, g] - p \bigl([x, \, g]\bigl)\\
f = f_p: V \times V \to \mathfrak{g}, \qquad f(x, y) &:=&
p \bigl([x, \, y]\bigl)\\
\{\, , \, \} = \{\, , \, \}_p: V \times V \to V, \qquad \{x, y\}
&:=& [x, \, y] - p \bigl([x, \, y]\bigl)
\end{eqnarray*}
for any $g \in \mathfrak{g}$ and $x$, $y\in V$. First of all, we
observe that the above maps are all well defined bilinear maps: $x
\triangleleft g \in V$ and $\{x, \, y \} \in V$, for all $x$, $y
\in V$ and $g \in \mathfrak{g}$. We shall prove that
$\Omega(\mathfrak{g}, V) = \bigl(\triangleleft, \, \triangleright,
\, f, \{-, \, -\} \bigl)$ is a Lie extending structure of
$\mathfrak{g}$ trough $V$ and
\begin{eqnarray*}
\varphi: \mathfrak{g} \,\natural \, V \to E, \qquad \varphi(g, x)
:= g+x
\end{eqnarray*}
is an isomorphism of Lie algebras that stabilizes $\mathfrak{g}$
and co-stabilizes $V$. Instead of proving the seven compatibility
conditions $(LE1)-(LE7)$, which requires a long and laborious
computation, we use the following trick combined with \thref{1}:
$\varphi: \mathfrak{g} \times V \to E$, $\varphi(g, \, x) := g+x$
is a linear isomorphism between the Lie algebra $E$ and the direct
product of vector spaces $\mathfrak{g} \times V$ with the inverse
given by $\varphi^{-1}(y) := \bigl(p(y), \, y - p(y)\bigl)$, for
all $y \in E$. Thus, there exists a unique Lie algebra structure
on $\mathfrak{g} \times V$ such that $\varphi$ is an isomorphism
of Lie algebras and this unique bracket on $\mathfrak{g} \times V$
is given by
$$
[(g, x), \, (h, y)] := \varphi^{-1} \bigl([\varphi(g, x), \,
\varphi(h, y)]\bigl)
$$
for all $g$, $h \in \mathfrak{g}$ and $x$, $y\in V$. The proof is
completely finished if we prove that this bracket coincides with
the one defined by \equref{brackunif} associated to the system
$\bigl(\triangleleft_p, \, \triangleright_p, \, f_p, \{-, \, -\}_p
\bigl)$. Indeed, for any $g$, $h \in \mathfrak{g}$ and $x$, $y\in
V$ we have:
\begin{eqnarray*}
[(g, x), \, (h, y)] &=& \varphi^{-1} \bigl([\varphi(g, x), \,
\varphi(h, y)]\bigl)
= \varphi^{-1} \bigl([g, \, h] + [g, \, y] + [x, \, h] + [x, \, y]\bigl)\\
&=& \bigl(p([g, \, h]), [g, \, h] - p([g, \, h])\bigl) +
\bigl(p([g, \, y]), [g, \, y] - p([g, \, y])\bigl)\\
&& + \bigl(p([x, \, h]), [x, \, h] - p([x, \, h])\bigl) +
\bigl(p([x, \, y]), [x, \, y] - p([x, \, ])\bigl)\\
&=& \Bigl(p([g, \, h]) + p([g, \, y]) + p([x, \, h]) +
p([x, \, y]), \ [g, \, h] + [g, \, y]\\
&&+ [x, \, h] + [x, \, y] - p([g, \, h]) - p([g, \, y])
- p([x, \, h]) - p([x, \, y])\Bigl)\\
&=& \Bigl([g, \, h] - y \triangleright g + x \triangleright h +
f(x, y), \,\{x, y\} + x \triangleleft h - y \triangleleft g\Bigl)
\end{eqnarray*}
as needed. Moreover, the following diagram is commutative
\begin{eqnarray*}
\xymatrix {& \mathfrak{g} \ar[r]^{i_{\mathfrak{g}}} \ar[d]_{Id} &
{\mathfrak{g} \,\natural \, V} \ar[r]^{q} \ar[d]^{\varphi} & V \ar[d]^{Id}\\
& \mathfrak{g} \ar[r]^{i} & {E}\ar[r]^{\pi } & V}
\end{eqnarray*}
where $\pi : E \to V$ is the projection of $E = \mathfrak{g} + V$
on the vector space $V$ and $q: {\mathfrak{g} \,\natural \, V} \to
V$, $q (g, x) := x$ is the canonical projection. The proof is now
finished.
\end{proof}

Using \thref{classif}, the classification of all Lie algebra
structures on $E$ that contains $\mathfrak{g}$ as a Lie
subalgebra, reduces to the classification of all unified products
$\mathfrak{g} \,\natural \, V$, associated to all Lie extending
structures $\Omega(\mathfrak{g}, V) = \bigl(\triangleleft, \,
\triangleright, \, f, \{-, \, -\} \bigl)$, for a given complement
$V$ of $\mathfrak{g}$ in $E$. In order to construct the
cohomological objects ${\mathcal H}^{2}_{\mathfrak{g}} \, (V, \,
\mathfrak{g} )$ and ${\mathcal H}^{2} \, (V, \, \mathfrak{g} )$
which will parameterize the classifying sets ${\rm Extd} \, (E,
\mathfrak{g})$ and respectively ${\rm Extd}' \, (E, \mathfrak{g})$
defined in \deref{echivextedn}, we need the following technical
lemma:

\begin{lemma} \lelabel{morfismuni}
Let $\Omega(\mathfrak{g}, V) = \bigl(\triangleleft, \,
\triangleright, \, f, \{-, \, -\} \bigl)$ and
$\Omega'(\mathfrak{g}, V) = \bigl(\triangleleft ', \,
\triangleright ', \, f', \{-, \, -\}' \bigl)$ be two Lie algebra
extending structures of $\mathfrak{g}$ trough $V$ and $
\mathfrak{g} \,\natural \, V$, $ \mathfrak{g} \,\natural \, ' V$
the associated unified products. Then there exists a bijection
between the set of all morphisms of Lie algebras $\psi:
\mathfrak{g} \,\natural \, V \to \mathfrak{g} \,\natural \, ' V$
which stabilizes $\mathfrak{g}$ and the set of pairs $(r, v)$,
where $r: V \to \mathfrak{g}$, $v: V \to V$ are two linear maps
satisfying the following compatibility conditions for any $g \in
\mathfrak{g}$, $x$, $y \in V$:
\begin{enumerate}
\item[(ML1)] $v(x) \triangleleft ' g = v(x \triangleleft g)$;
\item[(ML2)] $r(x \triangleleft g) = [r(x),\, g] - x
\triangleright g + v(x) \triangleright ' g$; \item[(ML3)] $v(\{x,
y\}) = \{v(x), v(y)\}' + v(x) \triangleleft ' r(y) - v(y)
\triangleleft ' r(x)$; \item[(ML4)] $r(\{x, y\}) = [r(x), \, r(y)]
+ v(x) \triangleright ' r(y) - v(y) \triangleright ' r(x) + f'
\bigl(v(x), v(y)\bigl) - f(x, y)$
\end{enumerate}
Under the above bijection the morphism of Lie algebras $\psi =
\psi_{(r, v)}: \mathfrak{g} \,\natural \, V \to \mathfrak{g}
\,\natural \, ' V$ corresponding to $(r, v)$ is given for any $g
\in \mathfrak{g}$ and $x \in V$ by:
$$
\psi(g, x) = (g + r(x), v(x))
$$
Moreover, $\psi = \psi_{(r, v)}$ is an isomorphism if and only if
$v: V \to V$ is an isomorphism and $\psi = \psi_{(r, v)}$
co-stabilizes $V$ if and only if $v = {\rm Id}_V$.
\end{lemma}

\begin{proof}
A linear map $\psi: \mathfrak{g} \,\natural \, V \to \mathfrak{g}
\,\natural \, ' V$ which makes the following diagram commutative:
$$
\xymatrix {& {\mathfrak{g}} \ar[r]^{i_{\mathfrak{g}}}
\ar[d]_{Id_{\mathfrak{g}}}
& {\mathfrak{g} \,\natural \, V}\ar[d]^{\psi}\\
& {\mathfrak{g}} \ar[r]^{i_{\mathfrak{g}}} & {\mathfrak{g}
\,\natural \, ' V}}
$$
is uniquely determined by two linear maps $r: V \to \mathfrak{g}$,
$v: V \to V$ such that $\psi(g, x) = (g + r(x), v(x))$, for all $g
\in \mathfrak{g}$, and $x \in V$. Indeed, if we denote $\psi(0, x)
= (r(x), v(x)) \in \mathfrak{g} \times V$ for all $x \in V$, we
have:
\begin{eqnarray*}
\psi(g, x) &=& \psi \bigl((g, 0) + \psi(0, x)\bigl) = \psi(g, 0) + \psi(0, x)\\
&=&(g, 0) + \bigl(r(x), v(x)\bigl) = \bigl(g + r(x), v(x) \bigl)
\end{eqnarray*}
Let $\psi = \psi_{(r, v)}$ be such a linear map, i.e. $\psi(g, x) =
(g + r(x), v(x))$, for some linear maps $r: V \to \mathfrak{g}$,
$v: V \to V$. We will prove that $\psi$ is a morpshism of Lie
algebras if and only if the compatibility conditions $(ML1)-(ML4)$
hold. It is enough to prove that the compatibility
\begin{equation}\eqlabel{Liemap}
\psi \bigl([(g, x), \, (h, y)] \bigl) = [\psi(g, x), \, \psi(h,
y)]
\end{equation}
holds for all generators of $\mathfrak{g} \,\natural \, V$. First
of all, it is easy to see that \equref{Liemap} holds for the pair
$(g, 0)$, $(h, 0)$, for all $g$, $h \in \mathfrak{g}$. Now we
prove that \equref{Liemap} holds for the pair $(g, 0)$, $(0, x)$
if and only if $(ML1)$ and $(ML2)$ hold. Indeed, $\psi \bigl([(g,
0), \, (0, x)] \bigl) = [\psi(g, 0), \, \psi(0, x)]$ it is
equivalent to $\psi(- x \triangleright g, - x \triangleleft g) =
[(g, 0), \, (r(x), v(x))]$ and hence to $(- x \triangleright g +
r(- x \triangleleft g), v(- x \triangleleft g)) = ([g, \, r(x)] -
v(x) \triangleright ' g, - v(x) \triangleleft ' g)$, i.e. to the
fact that $(ML1)$ and $(ML2)$ hold. In the same manner it can be
proved that \equref{Liemap} holds for the pair $(0, x)$, $(g, 0)$
if and only if $(ML1)$ and $(ML2)$ hold.

Next, we prove that \equref{Liemap} holds for the pair $(0, x)$,
$(0, y)$ if and only if $(ML3)$ and $(ML4)$ hold. Indeed, $\psi
\bigl([(0, x), \, (0, y)] \bigl) = [\psi(0, x), \, \psi(0, y)]$ it
is equivalent to $\psi(f(x, y), \{x, y\}) = [(r(x), v(x)), \,
(r(y), v(y))]$; therefore it is equivalent to: $\bigl(f(x, y) + r(\{x, y\}), v(\{x,
y\})\bigl) = \bigl([r(x), \, r(y)] + v(x) \triangleright ' r(y) -
v(y) \triangleright ' r(x) + f'(v(x), v(y)), \{v(x), v(y)\}' +
v(x) \triangleleft ' r(y) - v(y) \triangleleft ' r(x)\bigl)$, i.e.
to the fact that $(ML3)$ and $(ML4)$ hold.

Assume now that $v: V \to V$ is bijective. Then $\psi_{(r, v)}$ is
an isomorphism of Lie algebras with the inverse given for any
$h \in \mathfrak{g}$ and $y \in V$ by:
$$
\psi_{(r, v)}^{-1}(h, y) = \bigl(h - r(v^{-1}(y)), v^{-1}(y)\bigl)
$$
Conversely, assume that $\psi_{(r, v)}$ is bijective. It follows
easily that $v$ is surjective. Thus, we are left to prove that $v$
is injective. Indeed, let $x \in V$ such that $v(x) = 0$. We have
 $\psi_{(r, v)}(0, 0) = (0, 0) = (0, v(x)) = \psi_{(r, v)}(-
r(x), x)$, and hence we obtain $x = 0$, i.e. $v$ is a bijection.
The last assertion is trivial and the proof is now finished.
\end{proof}

\begin{definition}\delabel{echiaa}
Let $\mathfrak{g}$ be a Lie algebra and $V$ a $k$-vector space.
Two Lie algebra extending structures of $\mathfrak{g}$ by $V$,
$\Omega(\mathfrak{g}, V) = \bigl(\triangleleft, \, \triangleright,
\, f, \{-, \, -\} \bigl)$ and $\Omega'(\mathfrak{g}, V) =
\bigl(\triangleleft ', \, \triangleright ', \, f', \{-, \, -\}'
\bigl)$ are called \emph{equivalent}, and we denote this by
$\Omega(\mathfrak{g}, V) \equiv \Omega'(\mathfrak{g}, V)$, if
there exists a pair $(r, v)$ of linear maps, where $r: V \to
\mathfrak{g}$ and $v \in {\rm Aut}_{k}(V)$ such that
$\bigl(\triangleleft ', \, \triangleright ', \, f', \{-, \, -\}'
\bigl)$ is implemented from $\bigl(\triangleleft, \,
\triangleright, \, f, \{-, \, -\} \bigl)$ using $(r, v)$ via:
\begin{eqnarray*}
x \triangleleft ' g &=& v \bigl(v^{-1}(x) \triangleleft g\bigl)\\
x \triangleright ' g &=& r \bigl(v^{-1}(x) \triangleleft g\bigl) +
v^{-1}(x) \triangleright g + [g, \, r\bigl(v^{-1}(x)\bigl)]\\
f'(x, y) &=& f \bigl(v^{-1}(x), v^{-1}(y)\bigl) + r \bigl(\{v^{-1}(x),
v^{-1}(y)\}\bigl) + [r\bigl(v^{-1}(x)), \, r\bigl(v^{-1}(y)\bigl)]\\
&& - \, r\bigl(v^{-1}(x) \triangleleft r
\bigl(v^{-1}(y)\bigl)\bigl) - v^{-1}(x) \triangleright r
\bigl(v^{-1}(y)\bigl)\bigl)
+  r\bigl(v^{-1}(y) \triangleleft r \bigl(v^{-1}(x)\bigl)\bigl)\\
&&  + \, v^{-1}(y) \triangleright r \bigl(v^{-1}(x)\bigl)\bigl) \\
\{x, y\}' &=& v \bigl(\{v^{-1}(x), v^{-1}(y)\}\bigl) - v
\bigl(v^{-1}(x) \triangleleft r \bigl(v^{-1}(y)\bigl)\bigl)
 + v \bigl(v^{-1}(y) \triangleleft r \bigl(v^{-1}(x)\bigl)\bigl)
\end{eqnarray*}
for all $g \in \mathfrak{g}$, $x$, $y \in V$.
\end{definition}

We recall from \deref{echivextedn} that ${\rm Extd} \, (E,
\mathfrak{g})$ denotes the set of all equivalence classes of Lie
algebra structures on $E$ which stabilizes $\mathfrak{g}$. As a
conclusion of this section the main result of this paper, which
gives the theoretical answer to the extending structure problem,
follows:

\begin{theorem}\thlabel{main1}
Let $\mathfrak{g}$ be a Lie algebra, $E$ a vector space that
contains $\mathfrak{g}$ as a subspace and $V$ a complement of
$\mathfrak{g}$ in $E$. Then:

$(1)$ $\equiv$ is an equivalence relation on the set ${\mathcal L}
(\mathfrak{g}, V)$ of all Lie extending structures of
$\mathfrak{g}$ through $V$. We denote by ${\mathcal
H}^{2}_{\mathfrak{g}} \, (V, \, \mathfrak{g} ) := {\mathcal L}
(\mathfrak{g}, V)/ \equiv $, the pointed quotient set.

$(2)$ The map
$$
{\mathcal H}^{2}_{\mathfrak{g}} \, (V, \, \mathfrak{g} ) \to {\rm
Extd} \, (E, \mathfrak{g}), \qquad \overline{(\triangleleft,
\triangleright, f, \{-, \, -\})} \rightarrow \bigl(\mathfrak{g}
\,\natural \, V, \, [ -  , \, - ] \bigl)
$$
is bijective, where $\overline{(\triangleleft, \triangleright, f,
\{-, \, -\})}$ is the equivalence class of $(\triangleleft,
\triangleright, f, \{-, \, -\})$ via $\equiv$.
\end{theorem}

\begin{proof} The proof follows from \thref{1},
\thref{classif} and \leref{morfismuni} once we observe that
$\Omega(\mathfrak{g}, V) \equiv \Omega'(\mathfrak{g}, V)$ in the
sense of \deref{echiaa} if and only if there exists an isomorphism
of Lie algebras $\psi: \mathfrak{g} \,\natural \, V \to
\mathfrak{g} \,\natural \, ' V$ which stabilizes $\mathfrak{g}$.
Therefore, $\equiv$ is
 an equivalence relation on the set ${\mathcal L} (\mathfrak{g}, V)$
of all Lie algebra extending structures $\Omega(\mathfrak{g}, V)$
and the conclusion follows from \thref{classif} and
\leref{morfismuni}.
\end{proof}

\begin{remark}\relabel{aldoileaob}
The second cohomological object ${\mathcal H}^{2} \, (V, \,
\mathfrak{g} )$ that parameterizes  ${\rm Extd}' \, (E,
\mathfrak{g})$ is constructed in a simple manner as
follows: two Lie algebra extending structures
$\Omega(\mathfrak{g}, V) = \bigl(\triangleleft, \, \triangleright,
\, f, \{-, \, -\} \bigl)$ and $\Omega'(\mathfrak{g}, V) =
\bigl(\triangleleft ', \, \triangleright ', \, f', \{-, \, -\}'
\bigl)$ are called \emph{cohomologous}, and we denote this by
$\Omega(\mathfrak{g}, V) \approx \Omega'(\mathfrak{g}, V)$ if and
only if $\triangleleft' = \triangleleft$ and there exists a linear
map $r: V \to \mathfrak{g}$ such that
\begin{eqnarray*}
x \triangleright ' g &=& x \triangleright g +
r\bigl(x \triangleleft g\bigl) - [r(x), \, g]\\
f'(x, y) &=& f (x, y) + r \bigl(\{x, \, y\}\bigl) + [r(x), \, r(y)] + \\
&&  + \, y \triangleright r (x) - x \triangleright r (y)
+ r\bigl(y \triangleleft r (x)\bigl) - r\bigl(x \triangleleft r (y) \bigl) \\
\{x, y\}' &=& \{x, y\} - x \triangleleft r (y) + y \triangleleft
r(x)
\end{eqnarray*}
for all $g \in \mathfrak{g}$, $x$, $y \in V$.

Similar to the proof of \thref{main1} we can easily see that
$\Omega(\mathfrak{g}, V) \approx \Omega'(\mathfrak{g}, V)$ if and
only if there exists an isomorphism of Lie algebras $\psi:
\mathfrak{g} \,\natural \, V \to \mathfrak{g} \,\natural \, ' V$
which stabilizes $\mathfrak{g}$ and co-stabilizes $V$. Thus,
$\approx$ is an equivalence relation on the set ${\mathcal L}
(\mathfrak{g}, V)$ of all Lie extending structures of
$\mathfrak{g}$ through $V$. If we denote ${\mathcal H}^{2} \, (V,
\, \mathfrak{g} ) := {\mathcal L} (\mathfrak{g}, V)/ \approx $,
the map
$$
{\mathcal H}^{2} \, (V, \, \mathfrak{g} ) \to {\rm Extd}' \, (E,
\mathfrak{g}), \qquad \overline{(\triangleleft, \triangleright, f,
\{-, \, -\})} \rightarrow \bigl(\mathfrak{g} \,\natural \, V, \, [
-  , \, - ] \bigl)
$$
is a bijection between ${\mathcal H}^{2} \, (V, \mathfrak{g})$ and
the isomorphism classes of all Lie algebra structures on $E$ which
stabilizes $\mathfrak{g}$ and co-stabilizes $V$.
\end{remark}

\section{Special cases of unified products}\selabel{cazurispeciale}
In this section we show that crossed products and bicrossed
products of two Lie algebras are both special cases of unified
products. We make the following convention: if one of the maps
$\triangleleft$, $\triangleright$, $f$ or $\{-, \, -\}$ of an
extending datum $\Omega(\mathfrak{g}, V) = \bigl(\triangleleft, \,
\triangleright, \, f, \{-, \, -\} \bigl)$ is trivial then we will
omit it from the quadruple $\bigl(\triangleleft, \,
\triangleright, \, f, \{-, \, -\} \bigl)$.

\subsection*{Crossed products and the extension problem}
Let $\Omega(\mathfrak{g}, V) = \bigl(\triangleleft, \,
\triangleright, \, f, \{-, \, -\} \bigl)$ be an extending datum of
$\mathfrak{g}$ through $V$ such that $\triangleleft$ is the
trivial map, i.e. $x \triangleleft g = 0$, for all $x$, $y\in V$
and $g \in \mathfrak{g}$. Then, $\Omega(\mathfrak{g}, V) =
\bigl(\triangleleft, \, \triangleright, \, f, \{-, \, -\} \bigl) =
\bigl(\triangleright, \, f, \{-, \, -\} \bigl) $ is a Lie
extending structure of $\mathfrak{g}$ through $V$ if and only if
$(V, \{-, -\})$ is a Lie algebra and the following compatibilities
hold for any $g$, $h\in \mathfrak{g}$ and $x$, $y$, $z\in V$:
\begin{eqnarray*}
&\bullet &  f(x,\, x) = 0\\
&\bullet &  x \triangleright [g, \, h] = [x \triangleright g,
\, h] + [g, \, x \triangleright h] \\
&\bullet &  \{x, \, y \} \triangleright g = x \triangleright (y
\triangleright g) - y \triangleright (x \triangleright g) + [g,
\, f(x, \, y)]\\
&\bullet & f(x, \{y, z \}) + f(y, \{z, x \}) + f(z, \{x, y \}) +
x\triangleright f(y, z) + y\triangleright f(z, x) +
z\triangleright f(x, y)=0
\end{eqnarray*}
In this case, the associated unified product $\mathfrak{g}
\,\natural \,_{\Omega(\mathfrak{g}, V)} V = \mathfrak{g}
\#_{\triangleleft}^f \, V $ is the \emph{crossed product} of the
Lie algebras $\mathfrak{g}$ and $V$. A system $( \mathfrak{g}, V,
\triangleright, f)$ consisting of two Lie algebras $\mathfrak{g}$,
$V$ and two bilinear maps $\triangleright : V \times \mathfrak{g}
\to \mathfrak{g}$, $f: V\times V \to \mathfrak{g}$ satisfying the
above four compatibility conditions will be called a \emph{crossed
system of Lie algebras}. The crossed product associated to the
crossed system $(\mathfrak{g}, V, \triangleright, f)$ is the Lie
algebra defined as follow: $\mathfrak{g} \#_{\triangleleft}^f \, V
= \mathfrak{g} \times \, V $ with the bracket given for any $g$,
$h \in \mathfrak{g}$ and $x$, $y \in V$ by:

\begin{equation}\eqlabel{brackcrosspr}
[(g, x), \, (h, y)] := \bigl( [g, \, h] + x \triangleright h -
y\triangleright g + f(x, y), \, \{x, \, y \} \bigl)
\end{equation}

The crossed product of Lie algebras provides the answer to the
following restricted version of the extending structures problem:
\emph{Let $\mathfrak{g}$ be a Lie algebra, $E$ a vector space
containing $\mathfrak{g}$ as a subspace. Describe and classify all
Lie algebra structures on $E$ such that $\mathfrak{g}$ is an ideal
of $E$.}

Indeed, let $(\mathfrak{g}, V, \triangleright, f)$ be a crossed
system of two Lie algebras. Then, $\mathfrak{g} \cong \mathfrak{g}
\times \{0\}$ is an ideal in the Lie algebra $\mathfrak{g}
\#_{\triangleleft}^f \, V$ since $[(g, 0), \, (h, y)] := \bigl(
[g, \, h] - y\triangleright g , \, 0 \bigl)$. Conversely, crossed
products describe all Lie algebra structures on a vector space $E$
such that a given Lie algebra $\mathfrak{g}$ is an ideal of $E$.

\begin{corollary}\colabel{croslieide}
Let $\mathfrak{g}$ be a Lie algebra, $E$ a vector space containing
$\mathfrak{g}$ as a subspace. Then any Lie algebra structure on
$E$ that contains $\mathfrak{g}$ as an ideal is isomorphic to a
crossed product of Lie algebras $\mathfrak{g} \#_{\triangleleft}^f
\, V$.
\end{corollary}

\begin{proof}
Let $[-,\, -]$ be a Lie algebra structure on $E$ such that
$\mathfrak{g}$ is an ideal in $E$. In particular, $\mathfrak{g}$
is a subalgebra of $E$ and hence we can apply \thref{classif}. In
this case the action $\triangleleft = \triangleleft_p$ of the Lie
extending structure $\Omega(\mathfrak{g}, V) =
\bigl(\triangleleft_p, \, \triangleright_p, \, f_p, \{-, \, -\}_p
\bigl)$ constructed in the proof of \thref{classif} is the trivial
map since for any $x \in V$ and $g \in \mathfrak{g}$ we have that
$[x, g] \in \mathfrak{g}$ and hence $p ([x, g]) = [x, g]$. Thus,
$x \triangleleft_p g = 0$, i.e. the unified product $\mathfrak{g}
\,\natural \,_{\Omega(\mathfrak{g}, V)} V = \mathfrak{g}
\#_{\triangleleft}^f \, V $ is the crossed product of the Lie
algebras $\mathfrak{g}$ and $V:= {\rm Ker}(p)$.
\end{proof}

\begin{remark}\relabel{conecurburi}
We have proved \coref{croslieide} based on \thref{classif} by
taking a linear retraction $p$ of the canonical inclusion $i:
\mathfrak{g} \hookrightarrow E$ and the crossed system arising
from $p$. The classical proof of \coref{croslieide}, given in the
extension theory of Lie algebras, is completely different by the
way the crossed system is constructed: since $\mathfrak{g}$ is an
ideal of the Lie algebra $E$ we can consider the quotient Lie
algebra $\mathfrak{h} := E/\mathfrak{g}$. Let $\pi : E \to
\mathfrak{h}$ be the canonical projection and $s : \mathfrak{h}
\to E$ be a linear section of $\pi$. We define the action
$\triangleright = \triangleright_s$ and the cocycle $f = f_s$
associated to $s$ by the formulas \equref{cro1} and \equref{cro2}
from the introduction. Then, $(\mathfrak{g}, \, \mathfrak{h}, \,
\triangleright = \triangleright_s, \, f = f_s)$ is a crossed
system of Lie algebras and the map $\psi : \mathfrak{g}
\#_{\triangleleft}^f \, \mathfrak{h} \to E$, $\psi (g, x) := g +
s(x)$ is an isomorphism of Lie algebras with the inverse
$\psi^{-1} (z) := (z - s (\pi (z)), \, \pi (z))$, for all $z \in
E$.
\end{remark}

The restricted version of the extending structures problem is in
fact an equivalent reformulation of the (non-abelian) extension
problem. Indeed, first of all we remark that any crossed product
$\mathfrak{g} \#_{\triangleleft}^f \, V$ is an extension of
$\mathfrak{g}$ by $V$ via the following sequence:
\begin{eqnarray*} \eqlabel{extencros}
\xymatrix{ 0 \ar[r] & \mathfrak{g} \ar[r]^{i} & {\mathfrak{g}
\#_{\triangleleft}^f \, V} \ar[r]^{\pi} & V \ar[r] & 0 }
\end{eqnarray*}
where $i : \mathfrak{g} \to \mathfrak{g} \#_{\triangleleft}^f \,
V$, $i (g) := (g, 0)$ and $\pi : \mathfrak{g} \#_{\triangleleft}^f
\, V \to V$, $\pi (g, x) := x$. Conversely, let $\mathfrak{E}$ be
an extension of $\mathfrak{g}$ by $\mathfrak{h}$, that is there
exists an exact sequence of Lie algebras of the form
\equref{extencros0}. By identifying $\mathfrak{g} \cong {\rm Im
(i) = {\rm Ker} (\pi)}$ we view $\mathfrak{g}$ as an ideal of
$\mathfrak{E}$. Then, it follows from \coref{croslieide} that
there exists a crossed system $(\mathfrak{g}, \mathfrak{h},
\triangleright, f)$ of Lie algebras such that $\mathfrak{g}
\#_{\triangleleft}^f \, \mathfrak{h} \cong \mathfrak{E}$, an
isomorphism of Lie algebras. Furthermore, using \thref{classif},
the isomorphism can be chosen such that it stabilizes
$\mathfrak{g}$ and co-stabilizes $\mathfrak{h}$.

\subsection*{The abelian case: cotangent extending structures}
Let $E$ be a vector space, $\mathfrak{g}$ a subspace of $E$ with
the abelian Lie algebra structure. A Lie algebra structure on $E$
containing $\mathfrak{g}$ as an ideal is called a \emph{cotangent
extending structure}. The terminology is a generalization of the
concept introduced in \cite[Section 3.1]{ov}. Let $V$ be a given
complement of $\mathfrak{g}$ in $E$: it follows from
\coref{croslieide} and \thref{classif} that the set of all
cotangent extending structures of the abelian Lie algebra
$\mathfrak{g}$ to the vector space $E$ is parameterized by the set
of all triples $(\triangleright, \, f, \, \{-, \, - \})$, such
that $(V, \{-, \, - \})$ is a Lie algebra, $(\mathfrak{g},
\triangleright)$ is a left $V$-module and $f: V\times V \to
\mathfrak{g}$ is a bilinear map such that $f(x, \, x) = 0$ and
$$
f\bigl(x, \{y,\, z \}\bigl) + f\bigl(y, \{z,\, x \}\bigl) +
f\bigl(z, \{x,\, y \}\bigl) + x \triangleright f(y,  z) + y
\triangleright f(z,  x) + z \triangleright f(x, y) = 0
$$
for all $x$, $y$, $z\in V$. For such a triple $(\triangleright, \,
f, \, \{-, \, - \})$, the bracket of the cotangent extending
structure on $E \cong \mathfrak{g} \times \, V $ is given by:
\begin{equation}\eqlabel{cotangent}
[(g, x), \, (h, y)] := \bigl(x \triangleright h - y\triangleright
g + f(x, y), \, \{x, \, y \} \bigl)
\end{equation}
for all $g$, $h \in \mathfrak{g}$ and $x$, $y \in V$. Moreover,
any cotangent bracket on $E$ has the form \equref{cotangent}.

\subsection*{Bicrossed products and the factorization problem}
Let $\Omega(\mathfrak{g}, V)=\bigl(\triangleleft, \triangleright,
f, \{-, \, -\}\bigl)$ be an extending datum of $\mathfrak{g}$
through $V$ such that $f$ is the trivial map, i.e. $f(x, y) = 0$,
for all $x$, $y\in V$. Then, $\Omega(\mathfrak{g}, V) =
\bigl(\triangleleft, \, \triangleright, \,  \{-, \, -\} \bigl)$ is
a Lie extending structure of $\mathfrak{g}$ through $V$ if and
only if $(V, \, \{-, -\})$ is a Lie algebra and $( \mathfrak{g},
V, \triangleleft, \, \triangleright)$ is a \emph{matched pair of
Lie algebras} as defined in \cite[Theorem 4.1]{majid} and
independently in \cite[Theorem 3.9]{LW}: i.e. $\mathfrak{g}$ is a
left $V$-module under $\triangleright: V \ot \mathfrak{g} \to
\mathfrak{g}$, $V$ is a right $\mathfrak{g}$-module under
$\triangleleft: V \ot \mathfrak{g} \to V$ and the following
compatibilities hold for all $g$, $h \in \mathfrak{g}$, $x$, $y
\in V$:
\begin{eqnarray}
x \triangleright [g, \, h] &=& [x \triangleright g, \,  h] + [g,
\, x \triangleright h] + (x \triangleleft g) \triangleright h - (x
\triangleleft h) \triangleright g \eqlabel{mpLie1} \\
\{x , \, y \} \triangleleft g &=& \{x,\, y \triangleleft g \} +
\{x \triangleleft g, \, y\} + x \triangleleft (y \triangleright g)
- y \triangleleft (x \triangleright g) \eqlabel{mpLie2}
\end{eqnarray}
In this case, the associated unified product $\mathfrak{g}
\,\natural \,_{\Omega(\mathfrak{g}, V)} V = \mathfrak{g} \bowtie V
$ is precisely the \emph{bicrossed product} of the matched pair $(
\mathfrak{g}, V, \triangleleft, \, \triangleright)$ of Lie
algebras. In this paper we adopt the name bicrossed product
established in group theory \cite{Takeuchi} and Hopf algebra
theory \cite{Kassel}. Other names used in the literature for the
above product are: \emph{bicrossproduct} in \cite[Theorem
4.1]{majid}, \emph{double cross sum} in \cite[Proposition
8.3.2]{majid2}, \emph{double Lie algebra} \cite[Definition
3.3]{LW} or \emph{knit product} in \cite{Mic}. Important examples
of bicrossed products of Lie algebras are Manin's triples
\cite[Definition 1.13 and Theorem 1.12]{LW}. The bicrossed product
of two Lie algebras is the construction which provides the answer
for the so-called \emph{factorization problem}, the dual of the
extension problem and is also a special case of the extending
structures problem:

\emph{Let $\mathfrak{g}$ and $\mathfrak{h}$ be two given Lie
algebras. Describe and classify all Lie algebras $\Xi$ that
factorize trough $\mathfrak{g}$ and $\mathfrak{h}$, i.e. $\Xi$
contains $\mathfrak{g}$ and $\mathfrak{h}$ as Lie subalgebras such
that $\Xi = \mathfrak{g} + \mathfrak{h}$ and $\mathfrak{g} \cap
\mathfrak{h} = \{0\}$.}

Now, a Lie algebra $\Xi$ factorizes through $\mathfrak{g}$ and
$\mathfrak{h}$ if and only if there exists a matched pair of Lie
algebras $(\mathfrak{g}, \mathfrak{h}, \, \triangleleft, \,
\triangleright)$ such that $ \Xi \cong \mathfrak{g} \bowtie
\mathfrak{h}$ \cite[Proposition 8.3.2]{majid2} or \cite[Theorem
3.9]{LW}. Thus, the factorization problem can be restated in a
purely computational manner: \emph{Let $\mathfrak{g}$ and
$\mathfrak{h}$ be two given Lie algebras. Describe the set of all
matched pairs $(\mathfrak{g}, \mathfrak{h}, \, \triangleleft, \,
\triangleright)$  and classify up to an isomorphism all bicrossed
products $\mathfrak{g} \bowtie \mathfrak{h}$}.

In the case that $k$ is algebraically closed of characteristic
zero and $\Xi$ is a finite dimensional Lie algebra then the famous
Levi-Malcev theorem \cite[Theorem 5]{bour} proves that there
exists a Lie subalgebra $\mathfrak{h}$ of $\Xi$, called a Levi
subalgebra, such that $\Xi$ factorizes through ${\rm Rad} (\Xi)$
and $\mathfrak{h}$, where ${\rm Rad} (\Xi)$ is the radical of
$\Xi$. Thus, any finite dimensional Lie algebra $\Xi$ is
isomorphic to a bicrossed product between ${\rm Rad} (\Xi)$ and a
semi-simple Lie algebra $\mathfrak{h} \cong \Xi/{\rm Rad} (\Xi)$.

\begin{remark} \relabel{complexproduct}
An interesting equivalent description for the factorization of a
Lie algebra $\Xi$ through two Lie subalgebras is proved in
\cite[Proposition 2.2]{AS}. A linear map $f: \Xi \to \Xi$ is
called a \emph{complex product structure} on $\Xi$
\cite[Definition 2.1]{ABDO} if $f \neq \pm {\rm Id}$, $f^2 = f$
and $f$ is \emph{integrable}, that is for any $x$, $y \in \Xi$ we
have:
$$
f ([x, \, y] ) = [ f(x), \, y] + [x, \, f(y)] - f \bigl( [f(x), \,
f(y)] \bigl)
$$
If the characteristic of $k$ is $\neq 2$, then the linear map $f :
\mathfrak{g} \bowtie \mathfrak{h} \to \mathfrak{g} \bowtie
\mathfrak{h}$, $f (g, h) := (g, - h)$, for all $g \in
\mathfrak{g}$ and $h \in \mathfrak{h}$ is a complex product
structure on any bicrossed product $\mathfrak{g} \bowtie
\mathfrak{h}$ of Lie algebras. Conversely, if $f$ is a complex
product structure on $\Xi$, then $\Xi$ factorizes through two Lie
subalgebras $\Xi = \Xi_{+} + \Xi_{+}$, where $\Xi_{\pm}$ denote
the eigenspaces corresponding to the eigenvalue $\pm 1$ of $f$
\cite[Proposition 2.2]{AS}.
\end{remark}

\section{Flag extending structures. Examples}\selabel{exemple}

\thref{main1} offers the theoretical answer to the extending
structures problem. The next challenge is a computational one: for
a given Lie algebra $\mathfrak{g}$ that is a subspace in a vector
space $E$ with a given complement $V$ to compute explicitly the
classifying object ${\mathcal H}^{2}_{\mathfrak{g}} \, (V, \,
\mathfrak{g} )$ and then to list all Lie algebra structures on $E$
which extend the Lie algebra structure on $\mathfrak{g}$. In what
follows we provide a way of answering this problem for a large
class of such structures.

\begin{definition} \delabel{flagex}
Let $\mathfrak{g}$ be a Lie algebra and $E$ a vector space
containing $\mathfrak{g}$ as a subspace. A Lie algebra structure
on $E$ such that $\mathfrak{g}$ is a Lie subalgebra is called a
\emph{flag extending structure} of $\mathfrak{g}$ to $E$ if there
exists a finite chain of Lie subalgebras of $E$
\begin{equation} \eqlabel{lant}
\mathfrak{g} = E_0 \subset E_1 \subset \cdots \subset E_m = E
\end{equation}
such that $E_i$ has codimension $1$ in $E_{i+1}$, for all $i = 0,
\cdots, m-1$.
\end{definition}

In the context of \deref{flagex} we have that ${\rm dim}_k (V) =
m$, where $V$ is the complement of $\mathfrak{g}$ in $E$. The
existence of such a chain of Lie subalgebras is quite common in
the theory of solvable Lie algebras \cite[Proposition 2]{bour},
\cite[Lie Theorem]{H}. All flag extending structures of
$\mathfrak{g}$ to $E$ can be completely described by a recursive
reasoning where the key step is $m = 1$. More precisely, this key
step describes and classifies all unified products $\mathfrak{g}
\,\natural \, V_1$, for a $1$-dimensional vector space $V_1$. We
will prove that they are parameterized by the space ${\rm TwDer}
(\mathfrak{g})$ of all twisted derivations of $\mathfrak{g}$.
Then, by replacing the initial Lie algebra $\mathfrak{g}$ with
such a unified product $\mathfrak{g} \,\natural \, V$ which can be
described in terms of $\mathfrak{g}$ only, we can iterate the
process: in this way, on our second step we describe and classify
all unified products of the form $(\mathfrak{g} \,\natural \, V_1)
\,\natural \, V_2$, where $V_1$ and $V_2$ are vector spaces of
dimension $1$. Of course, after $m = {\rm dim}_k (V)$ steps, we
obtain the description of all flag extending structures of
$\mathfrak{g}$ to $E$. We start by introducing the following
concept which generalizes the notion of derivation of a Lie
algebra:

\begin{definition}\delabel{lambdaderivariii}
A \emph{twisted derivation} of a Lie algebra $\mathfrak{g}$ is a
pair $(\lambda, D)$ consisting of two linear maps $\lambda :
\mathfrak{g} \to k$ and $D : \mathfrak{g} \to \mathfrak{g}$ such
that for any $g$, $h\in \mathfrak{g}$:
\begin{eqnarray}
\lambda ([g, \, h]) &=& 0 \eqlabel{lamdaderiv0}\\
D ([g, \, h]) &=& [D(g), \, h] + [g, \, D(h)] + \lambda(g) D (h) -
\lambda(h) D(g)\eqlabel{lambderivari}
\end{eqnarray}
\end{definition}

The set of all twisted derivations of $\mathfrak{g}$ will be
denoted by ${\rm TwDer} (\mathfrak{g})$. The compatibility
condition \equref{lamdaderiv0} is equivalent to $\mathfrak{g}'
\subseteq {\rm Ker} (\lambda)$, where $\mathfrak{g}'$ is the
derived algebra of $\mathfrak{g}$.

\begin{examples} \exlabel{exemplederviciud}
1. ${\rm TwDer} (\mathfrak{g})$ contains the usual space of
derivations ${\rm Der} (\mathfrak{g})$ via the canonical embedding
$$
{\rm Der} (\mathfrak{g}) \hookrightarrow {\rm TwDer}
(\mathfrak{g}), \qquad D \mapsto (0, D)
$$
which is an isomorphism if $\mathfrak{g}$ is a perfect Lie
algebra.

2. Let $g_0 \in \mathfrak{g}$ and $\lambda : \mathfrak{g} \to k$
be a $k$-linear map such that $\mathfrak{g}' \subseteq {\rm Ker}
(\lambda)$. We define the map
$$
D_{g_0, \lambda} : \mathfrak{g} \to \mathfrak{g}, \quad D_{g_0,
\lambda} (h) := [g_0, \, h] - \lambda(h) g_0
$$
for all $h \in \mathfrak{g}$. Then $(\lambda, D_{g_0, \lambda} )$
is a twisted derivation called a \emph{inner twisted derivation}.
\end{examples}

We shall prove now that the set of all Lie extending structures
${\mathcal L} \, (\mathfrak{g}, V)$ of a Lie algebra
$\mathfrak{g}$ through a $1$-dimensional vector space $V$ is
parameterized by ${\rm TwDer} (\mathfrak{g})$.

\begin{proposition}\prlabel{unifdim1}
Let $\mathfrak{g}$ be a Lie algebra and $V$ a vector space of
dimension $1$ with a basis $\{x\}$. Then there exists a bijection
between the set ${\mathcal L} \, (\mathfrak{g}, V)$ of all Lie
extending structures of $\mathfrak{g}$ trough $V$ and the space
${\rm TwDer} (\mathfrak{g})$ of all twisted derivations of
$\mathfrak{g}$. Through the above bijection, the Lie extending
structure $\Omega(\mathfrak{g}, V)  = \bigl(\triangleleft, \,
\triangleright, f, \{-, -\} \bigl)$ corresponding to $(\lambda, D)
\in {\rm TwDer} (\mathfrak{g})$ is given by:
\begin{equation}\eqlabel{extenddim1}
x \triangleleft g = \lambda (g) x, \quad x \triangleright g =
D(g), \quad f = 0, \quad \{-, \, -\} = 0
\end{equation}
for all $g \in \mathfrak{g}$. The unified product associated to
the Lie extending structure \equref{extenddim1} will be denoted by
$\mathfrak{g} \,\natural \,_{(\lambda, \, D)} V$ and has the
bracket given for any $g$, $h\in \mathfrak{g}$ by:\footnote{As
usually, we define the bracket only in the points where the values
are non-zero.}
\begin{equation}\eqlabel{extenddim20}
[(g, 0), \, (h, 0)] = ([g, h], \, 0), \quad [(g, 0), \, (0, x)] =
- ( D(g),  \, \lambda(g) x)
\end{equation}
\end{proposition}

\begin{proof}
We have to compute all bilinear maps $ \triangleleft : V \times
\mathfrak{g} \to \mathfrak{g}$, $\triangleright : V \times
\mathfrak{g} \to V$, $f: V\times V \to \mathfrak{g}$ and $\{-, \,
-\} : V\times V \to V$ satisfying the compatibility conditions
$(LE1)-(LE7)$ of \thref{1}. Since, $f: V \times V \to
\mathfrak{g}$ and $\{-, \, -\} : V \times V \to V$ are bilinear
maps and $V$ has dimension $1$, the axiom $(LE1)$ is equivalent to
the fact that both maps are trivial, i.e. $f = 0$ and $\{-, \, -\}
= 0$. Again by the fact that ${\rm dim}_k (V) = 1$, we obtain a
bijection between the set of all bilinear maps $\triangleleft : V
\times \mathfrak{g} \to V$ and the set of all linear maps $\lambda
: \mathfrak{g} \to k$ and the bijection is given such that the
action $\triangleleft : V \times \mathfrak{g} \to V$ associated to
$\lambda$ is given by the formula: $ x \triangleleft g := \lambda
(g) x$, for all $g\in \mathfrak{g}$. Analogous, any bilinear map
$\triangleright : V\times \mathfrak{g} \to \mathfrak{g}$ is
uniquely implemented by a linear map $D = D_{\triangleright} :
\mathfrak{g} \to \mathfrak{g}$ via the formula: $x \triangleright
g := D(g)$, for all  $g\in \mathfrak{g}$.

Finally, we are left to deal with the compatibilities
$(LE2)-(LE7)$. Since, ${\rm dim}_k (V) = 1$, $f = 0$, $\{-, \, -\}
= 0$, we observe that the compatibilities $(LE4)-(LE7)$ are
automatically satisfied. Now, the compatibility condition $(LE2)$
is equivalent to $\lambda ([g, \, h]) = 0$, for all $g$, $h \in
\mathfrak{g}$. Finally, the compatibility condition $(LE3)$ is
equivalent to the fact that \equref{lambderivari} holds and the
proof is now finished.
\end{proof}

\begin{remark}\relabel{unifbicross}
Since the cocycle $f$ in \equref{extenddim1} is trivial we obtain
that $\mathfrak{g} \,\natural \,_{(\lambda, \, D)} V =
\mathfrak{g} \bowtie V$, where $\mathfrak{g} \bowtie V$ is a
bicrossed product between $\mathfrak{g}$ and an abelian Lie
algebra of dimension $1$. Hence, \prref{unifdim1} shows in fact
that any unified product $\mathfrak{g} \,\natural \,_{(\lambda, \,
D)} V$ between an arbitrary Lie algebra $\mathfrak{g}$ and a
$1$-dimensional vector space $V$ is isomorphic to a bicrossed
product $\mathfrak{g} \bowtie V$ between $\mathfrak{g}$ and the
abelian Lie algebra of dimension $1$.
\end{remark}

Next, we classify all Lie algebras $\mathfrak{g} \,\natural
\,_{(\lambda, \, D)} V$ by computing the cohomological objects
${\mathcal H}^{2}_{\mathfrak{g}} (V, \mathfrak{g})$ and ${\mathcal
H}^{2} \, (V, \mathfrak{g})$. First we need the following:

\begin{definition} \delabel{echivtwderivari}
Two twisted derivations $(\lambda, D)$ and $(\lambda', D') \in
{\rm TwDer} (\mathfrak{g})$ are called \emph{equivalent} and we
denote this by $(\lambda, D) \equiv (\lambda', D')$ if $\lambda =
\lambda'$ and there exists a pair $(g_0, q) \in \mathfrak{g}
\times k^*$ such that for any $h \in \mathfrak{g}$ we have:
\begin{equation}\eqlabel{echivtwder}
D(h) = q D' (h) + [g_0, \, h] - \lambda (h) g_0
\end{equation}
\end{definition}

Hence, $(\lambda, D) \equiv (\lambda', D')$ if and only if
$\lambda = \lambda'$ and there exists a non-zero scalar $q\in k$
such that $D - q D'$ is a inner twisted derivation. We provide
below the first explicit classification result of the extending
structures problem for Lie algebras. This is also the key step in
the classification of all flag extending structures.

\begin{theorem}\thlabel{clasdim1}
Let $\mathfrak{g}$ be a Lie algebra of codimension $1$ in the
vector space $E$ and $V$ a complement of $\mathfrak{g}$ in $E$.
Then:

$(1)$ $\equiv$ is an equivalence relation on the set ${\rm TwDer}
(\mathfrak{g})$ of all twisted derivations of $\mathfrak{g}$.

$(2)$ ${\rm Extd} \, (E, \mathfrak{g}) \cong {\mathcal
H}^{2}_{\mathfrak{g}} (V, \mathfrak{g} ) \cong {\rm TwDer}
(\mathfrak{g})/\equiv $. The bijection between ${\rm TwDer}
(\mathfrak{g})/\equiv$ and ${\rm Extd} \, (E, \mathfrak{g})$, the
isomorphisms classes of all Lie algebras structures on $E$ that
stabilizes $\mathfrak{g}$, is given by:
$$
\overline{(\lambda, \, D)} \mapsto \mathfrak{g} \,\natural
\,_{(\lambda, \, D)} V
$$
where $\overline{(\lambda, \, D)}$ is the equivalence class of
$(\lambda, \, D)$ via the relation $\equiv$ and $\mathfrak{g}
\,\natural \,_{(\lambda, \, D)} V$ is the Lie algebra constructed
in \equref{extenddim20}.

$(3)$ ${\mathcal H}^{2} \, (V, \mathfrak{g} ) \cong {\rm TwDer}
(\mathfrak{g})/ \approx $, where $\approx$ is the following
relation: $(\lambda, D) \approx (\lambda', D')$ if and only if
$\lambda = \lambda'$ and $D - D'$ is a inner twisted derivation of
$\mathfrak{g}$.
\end{theorem}

\begin{proof} Let $(\lambda, D)$, $(\lambda', D') \in
{\rm TwDer} (\mathfrak{g})$ be two twisted derivations of
$\mathfrak{g}$ and $\Omega(\mathfrak{g}, V) = \bigl(\triangleleft,
\, \triangleright, f, \{-, -\} \bigl)$ respectively $\Omega'
(\mathfrak{g}, V) = \bigl(\triangleleft', \, \triangleright', f',
\{-, -\}' \bigl)$ the corresponding Lie extending structure
constructed in \equref{extenddim1}. We will prove that $(\lambda,
D) \equiv (\lambda', D')$ if and only if there exists an
isomorphism of Lie algebras $\mathfrak{g} \,\natural \,_{(\lambda,
\, D)} V \cong \mathfrak{g} \,\natural \,_{(\lambda', \, D')} V$
that stabilizes $\mathfrak{g}$. This observation together with
\prref{unifdim1} and \thref{main1} will finish the proof.

Indeed, using \leref{morfismuni} we obtain that there exists an
isomorphism $\mathfrak{g} \,\natural \,_{(\lambda, \, D)} V \cong
\mathfrak{g} \,\natural \,_{(\lambda', \, D')} V$ of Lie algebras
which stabilizes $\mathfrak{g}$ if and only if there exists a pair
$(r, v)$, where $r: V \to \mathfrak{g}$ and $v: V \to V$ are
linear maps satisfying the compatibility conditions $(ML1) -
(ML4)$ and $v$ is bijective. Since ${\rm dim}_k (V) = 1$, any
linear map $r: V \to \mathfrak{g}$ is uniquely determined by an
element $g_0 \in \mathfrak{g}$ such that $r(x) = g_0$, where
$\{x\}$ is a basis in $V$. On the other hand, any automorphism $v$
of $V$ is uniquely determined by a non-zero scalar $q\in k^*$ such
such $v (x) = q x$. It remains to check the compatibility
conditions $(ML1) - (ML4)$, for this pair of maps $(r = r_{g_0}, v
= v_q)$. Since $f = f ' = 0$ and $ \{-, \, - \} = \{-, \, - \}' =
0$ in the corresponding Lie extending structure, we obtain that
the compatibility conditions $(ML3)$ and $(ML4)$ are trivially
fulfilled. Now, the compatibility condition $(ML1)$ is equivalent
to $q \lambda' (g) x = q \lambda (g) x$, for all $g \in
\mathfrak{g}$, i.e. to the fact that $\lambda = \lambda'$, since
$q \neq 0$. Finally, the compatibility condition $(ML1)$ takes the
following equivalent form:
$$
\lambda (g) g_0 = [g_0, \, g] - D(g) + q D' (g)
$$
for all $g\in \mathfrak{g}$, which is precisely
\equref{echivtwder} from \deref{echivtwderivari}. Thus, using
\leref{morfismuni} we have proved that $\mathfrak{g} \,\natural
\,_{(\lambda, \, D)} V \cong \mathfrak{g} \,\natural
\,_{(\lambda', \, D')} V$ (an isomorphism of Lie algebras that
stabilizes $\mathfrak{g}$) if and only if $(\lambda, D) \equiv
(\lambda', D')$ and the proof is finished.
\end{proof}

\thref{clasdim1} takes the following simplified form in the case
of perfect Lie algebras.

\begin{corollary}\colabel{clasdim1perf}
Let $\mathfrak{g}$ be a perfect Lie algebra of codimension $1$ in
the vector space $E$ and $V$ a complement of $\mathfrak{g}$ in
$E$. Then:

$(1)$ ${\rm Extd} \, (E, \mathfrak{g}) \cong  {\mathcal
H}^{2}_{\mathfrak{g}} (V, \mathfrak{g} ) \cong {\rm Der}
(\mathfrak{g})/\approx$, where $\approx$ is the equivalence
relation on ${\rm Der} (\mathfrak{g})$ defined by: $D \approx D'$
if and only if there exists $q\in k^*$ such that $D - q D'$ is an
inner derivation of $\mathfrak{g}$. The bijection between ${\rm
Der} (\mathfrak{g})/\approx$ and ${\rm Extd} \, (E, \mathfrak{g})$
is given by
$$
\overline{D} \mapsto \mathfrak{g} \,\natural \,_D V
$$
where $\overline{D}$ is the equivalence class of $D$ via the
relation $\approx$ and $\mathfrak{g} \,\natural \,_D V$ is the Lie
algebra constructed in \equref{extenddim20} for $\lambda = 0$.

$(2)$ ${\mathcal H}^{2} \, (V, \mathfrak{g} ) \cong {\rm Out}
(\mathfrak{g})$.
\end{corollary}

\begin{proof}
Follows directly from \thref{clasdim1} using
\exref{exemplederviciud}: for a perfect Lie algebra
$\mathfrak{g}$, we have that ${\rm TwDer} (\mathfrak{g}) = \{0\}
\times {\rm Der} (\mathfrak{g})$.
\end{proof}

\begin{remark} \relabel{derivariexter}
We recall that the space of outer derivations of a Lie algebra
$\mathfrak{g}$ is the quotient vector space ${\rm Out}
(\mathfrak{g}) := {\rm Der} (\mathfrak{g}) / {\rm
Inn}(\mathfrak{g})$. Thus, by definition, two derivations $D$, $D'
\in {\rm Der} (\mathfrak{g})$ are congruent and we denote this by
$D \sim D'$ if $D - D' \in {\rm Inn}(\mathfrak{g} )$ and then the
space ${\rm Out} (\mathfrak{g})$ is defined as the quotient via
this congruence relation, i.e. ${\rm Der} (\mathfrak{g}) /\sim $.
Now, two congruent derivations are equivalent in the sense of
\coref{clasdim1perf} but the converse does not hold. This means
that there exists a canonical surjection
$$
{\rm Out} (\mathfrak{g}) \twoheadrightarrow {\rm Der}
(\mathfrak{g})/\approx
$$
We also note that the classifying object ${\rm Der}
(\mathfrak{g})/\approx$ is just a pointed set as it does not carry
a group structure: it is straightforward to see that the following
possible group structure $\overline{D} + \overline{D'} =
\overline{D + D'}$ is not well defined on ${\rm Der}
(\mathfrak{g})/\approx$. However, this is not at all surprising
from the the point of view of non-abelian extension theory: the
classifying object is not a group anymore but a pointed set.
\end{remark}

Now, we indicate a class of Lie algebras $\mathfrak{g}$ such that
the classifying objects of \coref{clasdim1perf} are both
singletons. This class contains the semisimple Lie algebras or,
more generally, the sympathetic Lie algebras \cite{ben}. In
particular, it shows that for a semisimple Lie algebra
$\mathfrak{g}$ there exists, up to an isomorphism that stabilizes
$\mathfrak{g}$, a unique Lie algebra structure on a vector space
of dimension $1 + {\rm dim}_k (\mathfrak{g})$ which extends the
one on $\mathfrak{g}$: more precisely, this unique Lie algebra
structure is given by the direct product $\mathfrak{g} \times V$,
between $\mathfrak{g}$ and an abelian Lie algebra of dimension
$1$.

\begin{corollary}\colabel{clasdim1dihai}
Let $\mathfrak{g}$ be a Lie algebra of codimension $1$ in the
vector space $E$ and $V$ a complement of $\mathfrak{g}$ in $E$.
Assume that $\mathfrak{g}$ is perfect and ${\rm Der}
(\mathfrak{g}) = {\rm Inn} (\mathfrak{g})$. Then ${\mathcal
H}^{2}_{\mathfrak{g}} (V, \mathfrak{g} ) = {\mathcal H}^{2} (V,
\mathfrak{g} ) = 0$.
\end{corollary}

\begin{proof} We apply \coref{clasdim1perf}: since
${\rm Der} (\mathfrak{g}) = {\rm Inn} (\mathfrak{g})$ we obtain
that the space of outer derivations ${\rm Out} (\mathfrak{g}) =
0$, hence, so is ${\rm Der} (\mathfrak{g})/\approx$, being a
quotient of a null space.
\end{proof}

Next we provide two explicit examples for the above results by
computing ${\mathcal H}^{2}_{\mathfrak{g}} (V, \mathfrak{g} )$ and
then describing all Lie algebra structures which extend the Lie
algebra structure from $\mathfrak{g}$ to a vector space of
dimension $1 + {\rm dim}_k (\mathfrak{g})$. The detailed
computations are rather long but straightforward and can be
provided upon request. We start with the case when $\mathfrak{g}$
is a perfect Lie algebra which is not semisimple.

\begin{example} \exlabel{ultimulexperfect}
Let $k$ be a field of characteristic zero and $\mathfrak{g}$ be
the perfect $5$-dimensional Lie algebra with a basis $\{e_{1},
e_{2}, e_{3}, e_{4}, e_{5}\}$ and bracket given by:
\begin{eqnarray*}
[e_{1}, \, e_{2}] = e_{3}, \quad [e_{1}, \, e_{3}] = -2e_{1},
\quad [e_{1}, \, e_{5}] = [e_{3}, \, e_{4}] = e_{4}\\
\left[ e_{2}, \, e_{3} \right] = 2e_{2}, \quad \left[e_{2}, \,
e_{4} \right] = e_{5}, \quad \left[e_{3}, \, e_{5} \right] = -
e_{5}
\end{eqnarray*}
We shall compute the classifying object ${\rm Extd} \, (k^6,
\mathfrak{g})$ by proving that $ {\rm Extd} \, (k^6, \mathfrak{g})
\cong k^7/ \equiv$, where $\equiv$ is the equivalence relation on
$k^7$ defined by: $(a_1, \cdots, a_7) \equiv (a'_1, \cdots, a'_7)$
if and only if there exists $q \in k^{*}$ such that $a_{2} = q
a^{'}_{2}$ and $2 a_{7} - a_{1} = q (2 a^{'}_{7} - a^{'}_{1})$.

Indeed, by a rather long but straightforward computation it can be
proved that the space of derivations ${\rm Der}(\mathfrak{g})$
coincides with the space of all matrices from $\mathcal{M}_{5}(k)$
of the form:
\begin{eqnarray*}
A = \left( \begin{array}{ccccc} a_{1} & 0 & a_{6} & 0 & 0\\
0 & -a_{1} & -2a_{2} & 0 & 0\\
a_{2} & a_{4} & 0 & 0 & 0\\
a_{3} & 0 & a_{5} & a_{7} & a_{4}\\
0 & a_{5} & -a_{3} & a_{2} & (a_{7}-a_{1}) \end{array}\right)
\end{eqnarray*}
for all $a_1, \cdots, a_7 \in k$. Thus, any $6$-dimensional Lie
algebra that contains $\mathfrak{g}$ as a Lie subalgebra is
isomorphic to one of the following seven parameter Lie algebra
denoted by $\mathfrak{g}_{(a_1, \cdots, a_7)} (x) := \mathfrak{g}
\,\natural \, V $, which has the basis $\{e_{1}, e_{2}, e_{3},
e_{4}, e_{5}, x\}$ and bracket given by:
$$
[e_1, \, x] = - a_1 e_1 - a_2 e_3 - a_3 e_4, \quad [e_2, \, x] =
a_{1}e_{2} - a_4 e_3 - a_5 e_5
$$
$$
[e_3, \, x] = - a_6 e_1 + 2a_2 e_2 - a_5 e_4 + a_3 e_5, \,\, [e_4,
\, x] = - a_7 e_4 - a_2 e_5, \,\, [e_5, \, x] = - a_4 e_4 + (a_1 -
a_7) e_5
$$
for some scalars $a_1, \cdots, a_7 \in k$. Two such Lie algebras
$\mathfrak{g}_{(a_1, \cdots, a_7)} (x)$ and $\mathfrak{g}_{(a'_1,
\cdots, a'_7)} (x)$ are equivalent in the sense of
\deref{echivextedn} if and only if there exists $q \in k^{*}$ such
that $a_{2} = q a^{'}_{2}$ and $2 a_{7} - a_{1} = q (2 a^{'}_{7} -
a^{'}_{1})$, as needed.
\end{example}

Finally, we give an example in the case when $\mathfrak{g}$ is not
perfect.

\begin{example}\exlabel{cazulneperfect}
Let $k$ be a field of characteristic zero and $\textsf{gl}(2, k)$
the Lie algebra of all $2 \times 2$ matrices over $k$ with the
usual Lie bracket defined by:
\begin{eqnarray}
[e_{ij}, \, e_{kl}] = \delta_{jk} e_{il} - \delta_{il} e_{kj}
\end{eqnarray}
for all $1 \leq i, j \leq 2$, where $\delta$ is the Kronecker
delta and $e_{ij}$ the matrix units. We will compute the twisted
derivations for $\textsf{gl}(2, k)$. The $k$-linear maps $\lambda:
\textsf{gl}(2, k) \to k$ satisfying \equref{lamdaderiv0} are given
as follows:
\begin{eqnarray*}
\lambda(e_{11}) = \lambda(e_{22}) := q \in k, \qquad
\lambda(e_{12}) = \lambda(e_{21}) = 0
\end{eqnarray*}
Now depending on the values of $q$ we obtain, by a rather long but
straightforward computation, the following $k$-linear maps $D:
\textsf{gl}(2, k) \to \textsf{gl}(2, k)$ satisfying
\equref{lambderivari}:

\textbf{Case 1}: Suppose first that $q \notin \{0, \, 1, \, -1, \,
2\}$. Then the space of $k$-linear maps $D: \textsf{gl}(2, k) \to
\textsf{gl}(2, k)$ satisfying \equref{lambderivari} coincide with
the space of all matrices from $\mathcal{M}_{4}(k)$ of the form:
\begin{eqnarray*}
A = \left( \begin{array}{cccc} a_{1} & -(1-q)^{-1}a_{3} & -(1+q)a_{2} & a_{1} \\
a_{2} & q^{-1}(a_{4}-a_{1}) & 0 & (q-1)(q+1)^{-1}a_{2}\\
a_{3} & 0 & q^{-1}(a_{1}-a_{4}) & (q+1)(q-1)^{-1}a_{3}\\
a_{4} & (1-q)^{-1}a_{3} & (1+q)^{-1}a_{2} &
a_{4}\end{array}\right)
\end{eqnarray*}
for all $a_1, \cdots, a_4 \in k$. Thus, in this case any
$5$-dimensional Lie algebra that contains $\textsf{gl}(2, k)$ as a
Lie subalgebra is isomorphic to one of the following five
parameter Lie algebra denoted by $\textsf{gl}(2, k)_{(q, a_1,
\cdots, a_4)} (x)$, which has the basis $\{e_{11}, e_{12}, e_{21},
e_{22}, x\}$ and bracket given by:
\begin{eqnarray*}
[e_{11}, \, x] &=& -a_{1}e_{11} - a_{2}e_{12} - a_{3}e_{21} -
a_{4} e_{22} - qx\\
\left[e_{12}, \, x \right] &=& (1-q)^{-1}a_{3} e_{11} +
q^{-1}(a_{1}-a_{4})e_{12} -
(1-q)^{-1}a_{3} e_{22}\\
\left[e_{21}, \, x\right] &=& (1+q)^{-1}a_{2} e_{11} +
q^{-1}(a_{4}- a_{1}) e_{21} -
(1+q)^{-1}a_{2}e_{22}\\
\left[e_{22}, \, x\right] &=& -a_{1}e_{11} - (q-1)(q+1)^{-1}
a_{2}e_{12} - (q+1)(q-1)^{-1}a_{3}e_{21} - a_{4}e_{22} - qx
\end{eqnarray*}
Two such Lie algebras $\textsf{gl}(2, k)_{(q, a_1, \cdots, a_4)}
(x)$ and $\textsf{gl}(2, k)_{(q, a'_1, \cdots, a'_4)} (x)$ are
equivalent in the sense of \deref{echivextedn} if and only if
there exists $p \in k^{*}$ such that $a_{2} = p a^{'}_{2}$, $a_{3}
= p a^{'}_{3}$ and $a_{1} - a_{4} = p(a^{'}_{1} - a^{'}_{4})$.

\textbf{Case 2}: Assume that $q = 0$. Then the space of $k$-linear
maps $D: \textsf{gl}(2, k) \to \textsf{gl}(2, k)$ satisfying
\equref{lambderivari} coincide with the space of all matrices from
$\mathcal{M}_{4}(k)$ of the form:
\begin{eqnarray*}
A = \left( \begin{array}{cccc} a_{1} & - a_{3} & -a_{2} & a_{1} \\
a_{2} & a_{4} & 0 & -a_{2}\\
a_{3} & 0 & -a_{4} & -a_{3}\\
a_{1} & a_{3} & a_{2} & a_{1}\end{array}\right)
\end{eqnarray*}
for all $a_1, \cdots, a_5 \in k$. Thus, in this case any
$5$-dimensional Lie algebra that contains $\textsf{gl}(2, k)$ as a
Lie subalgebra has the basis $\{e_{11}, e_{12}, e_{21}, e_{22},
x\}$ and bracket given by:
\begin{eqnarray*}
[e_{11}, \, x] = -a_{1}e_{11} - a_{2}e_{12} - a_{3}e_{21} - a_{1}
e_{22}, \quad
\left[e_{12}, \, x \right] = a_{3} e_{11} -a_{4}e_{12} - a_{3}e_{22}\\
\left[e_{21}, \, x\right] = a_{2}e_{11} + a_{4}e_{21} - a_{2}
e_{22}, \quad \left[e_{22}, \, x\right] = -a_{1}e_{11} +
a_{2}e_{12} + a_{3}e_{21} - a_{1} e_{22}
\end{eqnarray*}
Two such Lie algebras are equivalent in the sense of
\deref{echivextedn} if and only if there exists $p \in k^{*}$ such
that $a_{1} = pa^{'}_{1}$.

\textbf{Case 3}: Assume that $q = 1$. Then the space of $k$-linear
maps $D: \textsf{gl}(2, k) \to \textsf{gl}(2, k)$ satisfying
\equref{lambderivari} coincide with the space of all matrices from
$\mathcal{M}_{4}(k)$ of the form:
\begin{eqnarray*}
A = \left( \begin{array}{cccc} a_{1} & a_{3} & -2^{-1}a_{2} & a_{1} \\
a_{2} & a_{4}-a_{1} & 0 & 0\\
0 & 0 & a_{1}-a_{4} & 2a_{3}\\
a_{4} & -a_{3} & 2^{-1}a_{2} & a_{4}\end{array}\right)
\end{eqnarray*}
for all $a_1, \cdots, a_4 \in k$. Thus, in this case any
$5$-dimensional Lie algebra that contains $\textsf{gl}(2, k)$ as a
Lie subalgebra has the basis $\{e_{11}, e_{12}, e_{21}, e_{22},
x\}$ and bracket given by:
\begin{eqnarray*}
[e_{11}, \, x] = -a_{1}e_{11} - a_{2}e_{12} - a_{4}e_{22} - x,
\quad
\left[e_{12}, \, x \right] = -a_{3} e_{11} +(a_{1}-a_{4})e_{12} + a_{3}e_{22}\\
\left[e_{21}, \, x\right] = 2^{-1}a_{2}e_{11} +
(a_{2}-a_{1})e_{21} - 2^{-1}a_{2} e_{22}, \quad \left[e_{22}, \,
x\right] = -a_{1}e_{11} -2a_{3}e_{21} - a_{4}e_{22} - x
\end{eqnarray*}
Two such Lie algebras are equivalent in the sense of
\deref{echivextedn} if and only if there exists $p \in k^{*}$ such
that $a_{2} = pa^{'}_{2}$, $a_{3} = p a^{'}_{3}$ and $a_{1} -
a_{4} = p(a^{'}_{1} - a^{'}_{4})$.

\textbf{Case 4}: Assume that $q = -1$. Then the space of
$k$-linear maps $D: \textsf{gl}(2, k) \to \textsf{gl}(2, k)$
satisfying \equref{lambderivari} coincide with the space of all
matrices from $\mathcal{M}_{4}(k)$ of the form:
\begin{eqnarray*}
A = \left( \begin{array}{cccc} a_{1} & -2^{-1}a_{3} & a_{2} & a_{1} \\
0 & a_{1}-a_{4} & 0 & 2a_{2}\\
a_{3} & 0 & a_{4}-a_{1} & 0\\
a_{4} & 2^{-1}a_{3} & -a_{2} & a_{4}\end{array}\right)
\end{eqnarray*}
for all $a_1, \cdots, a_4 \in k$. Thus, in this case any
$5$-dimensional Lie algebra that contains $\textsf{gl}(2, k)$ as a
Lie subalgebra has the basis $\{e_{11}, e_{12}, e_{21}, e_{22},
x\}$ and bracket given by:
\begin{eqnarray*}
[e_{11}, \, x] = -a_{1}e_{11} - a_{3}e_{21} - a_{4}e_{22} + x,
\quad
\left[e_{12}, \, x \right] = 2^{-1}a_{3} e_{11} +(a_{4}-a_{1})e_{12} - 2^{-1}a_{3} e_{22}\\
\left[e_{21}, \, x\right] = -a_{2}e_{11} + (a_{1}-a_{4})e_{21} +
a_{2} e_{22}, \quad \left[e_{22}, \, x\right] = -a_{1}e_{11}
-2a_{2}e_{12} - a_{4}e_{22} + x
\end{eqnarray*}
Two such Lie algebras are equivalent in the sense of
\deref{echivextedn} if and only if there exists $p \in k^{*}$ such
that $a_{3} = pa^{'}_{3}$, $a_{2} = p a^{'}_{2}$ and $a_{1} -
a_{4} = p(a^{'}_{1} - a^{'}_{4})$.

\textbf{Case 5}: Assume that $q = 2$. Then the space of $k$-linear
maps $D: \textsf{gl}(2, k) \to \textsf{gl}(2, k)$ satisfying
\equref{lambderivari} coincide with the space of all matrices from
$\mathcal{M}_{4}(k)$ of the form:
\begin{eqnarray*}
A = \left( \begin{array}{cccc} a_{1} & a_{3} & -3^{-1}a_{2} & a_{1} \\
a_{2} & 2^{-1}(a_{4}-a_{1}) & a_{5} & 3^{-1}a_{2}\\
a_{3} & 0 & 2^{-1}(a_{1}-a_{4}) & 3a_{3}\\
a_{4} & -a_{3} & 3^{-1}a_{2} & a_{4}\end{array}\right)
\end{eqnarray*}
for all $a_1, \cdots, a_5 \in k$. Thus, in this case any
$5$-dimensional Lie algebra that contains $\textsf{gl}(2, k)$ as a
Lie subalgebra has the basis $\{e_{11}, e_{12}, e_{21}, e_{22},
x\}$ and bracket given by:
\begin{eqnarray*}
[e_{11}, \, x] &=& -a_{1}e_{11} - a_{2}e_{12} - a_{3}e_{21}-
a_{4}e_{22} - 2x\\
\left[e_{12}, \, x \right] &=& -a_{3} e_{11} + 2^{-1}(a_{1}-a_{4})e_{12} + a_{3}e_{22}\\
\left[e_{21}, \, x\right] &=& 3^{-1}a_{2}e_{11} - a_{5}e_{12} +
2^{-1}(a_{4}-a_{1})e_{21} -
3^{-1}a_{2} e_{22}\\
\left[e_{22}, \, x\right] &=& -a_{1}e_{11} - 3^{-1}a_{2}e_{12} -
3a_{3}e_{21}- a_{4}e_{22} - 2x
\end{eqnarray*}
Two such Lie algebras are equivalent in the sense of
\deref{echivextedn} if and only if there exists $p \in k^{*}$ such
that $a_{2} = pa^{'}_{2}$, $a_{3} = pa^{'}_{3}$, $a_{5} = p
a^{'}_{5}$ and $a_{1} - a_{4} = p(a^{'}_{1} - a^{'}_{4})$.

Thus we have described the classifying object ${\rm Extd} \, (k^5,
\, \textsf{gl}(2, k)) \cong {\mathcal H}^{2}_{\textsf{gl}(2, k)}
\, (V, \, \textsf{gl}(2, k))$: it is equal to the disjoint union
of the five quotient spaces described above. This is easy to see
having in mind that if two twisted derivations $(\lambda_1, D_1)$
are equivalent in the sense of \coref{clasdim1perf} then
$\lambda_1 = \lambda_2$.
\end{example}

\end{document}